\documentclass[11pt]{article}
\usepackage{amsmath,amssymb,amsthm,graphicx,wrapfig,url}
\usepackage[all]{xy}
\newtheorem{theorem}{Theorem}[section]
\newtheorem{lemma}[theorem]{Lemma}
\newtheorem{definition}[theorem]{Definition}

\newtheorem{proposition}[theorem]{Proposition}

\newcommand{\manbound}{\partial_{\mathrm{m}}}

\DeclareMathOperator{\id}{id}
\DeclareMathOperator{\im}{Im}
\DeclareMathOperator{\const}{const}
\DeclareMathOperator{\dist}{dist}

\DeclareMathOperator{\Cone}{Cone}
\newcommand{\argmin}{\mathrm{arg\, min}}

\def\:{\colon}

\newcommand{\R}{\mathbb{R}}
\newcommand{\Q}{\mathbb{Q}}

\newcommand{\Z}{\mathbb{Z}}
\newcommand{\F}{\mathbb{F}}
\newcommand{\heading}[1]{\vspace{1ex}\par\noindent{\bf\boldmath #1}}
\def\indef#1{\emph{#1}}

\def\ZrFr{Z_r^{\mathrm{fr}}}
\def\Zrf{Z_{<r}(f)}

\def\gr{\pi}

\begin{document}

\markboth{P. Franek, M. Kr\v c\'al, H. Wager}{Persistence of Zero Sets}

\title{Persistence of Zero Sets}

\author{Peter Franek, Marek Kr\v c\'al \\
IST Austria
}
\maketitle

\begin{abstract}
We study robust properties of zero sets of  continuous maps $f\:X\to\R^n$. 
Formally, we analyze the family $\Zrf=\{g^{-1}(0):\,\,\|g-f\|<r\}$ of all zero sets of all continuous maps $g$ closer to $f$
than $r$ in the max-norm.
All of these sets are outside $A:=\{x: \,|f(x)|\geq r\}$ and we claim that $Z_{<r}(f)$ is fully determined by $A$ and an~element of
certain \emph{cohomotopy group} which (by a recent result) is computable whenever the dimension of $X$
is at most $2n-3$.

By considering all $r>0$ simultaneously, the pointed cohomotopy groups form a persistence module---a structure leading to 
persistence diagrams as in the case of \emph{persistent homology} or \emph{well groups}.
Eventually, we get a descriptor of persistent robust properties
of zero sets that has better descriptive power (Theorem A) and better computability status (Theorem B)
than the established well diagrams.\footnote{Admittedly, well diagrams cover broader settings than we do here, 
but their application as a property and descriptor of $\Zrf$  is the most studied one.}
Moreover, if we endow every point of each zero set with gradients of the perturbation, the robust description of the zero sets by elements of cohomotopy groups is in some sense the best possible
(Theorem C).
\end{abstract}

\section{Introduction}
Vector valued continuous maps $f\:X\to \R^n$ are ubiquitous in modeling phenomena in science and technology. Their zero sets  $f^{-1}(0)$ play often an~important role in those models. Vector fields can represent dynamical systems, and their zeros are their key property. Similarly, maps $X\to\R^n$ can represent measured
continuous physical quantities such as MRI or ultrasound scans and the preimages of points in $\R^n$ correspond to isosurfaces.  
In nonlinear optimization, the set of feasible solutions is described as the zero set $f^{-1}(0)$  of a given continuous map $f\:X\to\R^n$.

In practice, we often have only access to approximations of those maps.
Either they are sampled by imprecise measurements or inferred from models that only approximate reality. 
Thus we need to understand their zero sets in a robust way. 
This is formalized as follows.  For a continuous map \(f:X\to \R^n\) defined on a topological space $X$ 
and a robustness radius $r\in\R^+$  we define 
\begin{align*}
& Z_{<r}(f):=\{g^{-1}(0)\mid g\:X\to\R^n \text{ such that }\|f-g\|<r\}
\end{align*}
where \(\|\cdot\|\) is the max-norm with respect to some fixed norm $|\cdot|$ in \(\R^n\). 

\begin{wrapfigure}{r}{0.5\textwidth}
\begin{center}
\includegraphics[scale=1.1]{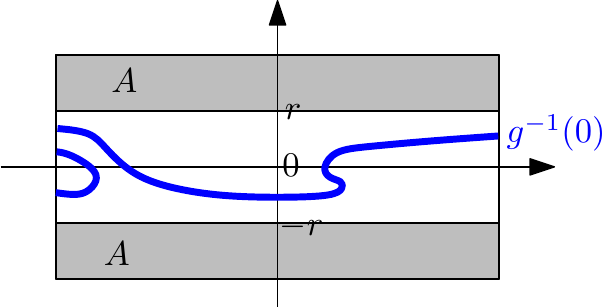}
\caption{The zero set of the scalar function $f(x,y)=y$ is the $x$-axis. 
For $r>0$, any $r$-perturbation $g$ has a~zero set $g^{-1}(0)$ that separates the two
components of $A:=\{x: \,|f(x)|\geq r\}$. Conversely, any closed set $Z\subseteq X$ disjoint from $A$ that separates these two components can be realized
as the zero set of some $r$-perturbation $g$ of $f$.}
\label{fig:fig-intro}
\end{center}
\end{wrapfigure} 
Any function $g$ with $\|g-f\|<r$ will be called an \emph{$r$-perturbation} and any  property of $f^{-1}(0)$ that is shared with 
$g^{-1}(0)$ for all $r$-perturbations $g$ is called an \emph{$r$-robust} property.
Invariants of zero sets that are preserved by $r$-perturbations translate to properties of $Z_{<r}(f)$: in particular, the problem 
of an $r$-robust \emph{existence} of zero translates to \emph{non-emptiness} of all sets in $Z_{<r}(f)$.

The problem $\emptyset\in Z_{<r}(f)$ has been analyzed from the algorithmic viewpoint when $X$ is a finite simplicial complex and $f$ is 
piecewise linear~\cite{nondec}. The results are surprising and far from obvious: the non-emptiness of all sets in $Z_{<r}(f)$ 
is algorithmically decidable if $\dim X \leq 2n-3$ or $n=1$ or $n$ is even.
Conversely, $\emptyset\in Z_{<r}(f)$ is algorithmically undecidable for odd $n\geq 3$. 
This has been shown by a reduction to the topological extension problem for maps to spheres
and to recent (un)decidability results for the latter~\cite{polypost,ext-hard,Krcal-thesis,vokrinek:oddspheres}.

However, \emph{non-emptiness} of all sets in $Z_{<r}(f)$ is only the simplest topological property, see Figure~\ref{fig:fig-intro} for a slightly more
interesting property.
Thus a~natural question is the following.

\medskip

``Which properties of zero set of $f$ are preserved under perturbations?''
\medskip

A notable attempt to attack this problem is the concept of \emph{well group}, based on studying
homological properties of zero sets. However, well groups do not constitute a~complete invariant of $Z_{<r}(f)$:  some properties
of zero sets are not captured by well groups, see~\cite[Thm. D,E]{well-group}.\footnote{Moreover, the computability of well groups is 
only known in some special cases.}
This paper is an attempt to answer the above question via means of homotopy theory. 
As we will see in Theorem C, under some mild assumptions, zero sets of smooth $r$-perturbations that are transverse to zero form
a~framed cobordism class of submanifolds of $X$. 
This suggests that homotopy theory is indeed the right tool for studying this problem and that homology alone is not sufficient.

\section{Statement of the results.}
\label{s:statements}

\heading{Robustness through lenses of homotopy theory.}
The  surprising recipe is not to analyze $f$ where its values are small, but rather where they are big---namely, of norm at least $r$. 
Therefore we need to refer to the set $A=\{x: |f(x)|\geq r\}$ on which any $r$-perturbation of $f$ is nonzero. 
Another surprising fact is that the analysis of $f|_A$ needs to be done only up to homotopy.\footnote{We say that maps $f,g\:X\to Y$ are homotopic whenever $f$ can be ``continuously deformed'' into $g$, that is, there is $H\:X\times[0,1]\to Y$ such that $H(\cdot,0)=f$ and $H(\cdot,1)=g$.} 
For, informally speaking, the notion of $r$-perturbation can be replaced by a corresponding notion of 
\emph{homotopy $r$-perturbation},{\footnote{A map $h\:X\to \R^n$ is a homotopy perturbation of $f$ whenever $h|_A$ is homotopic to $f|_A$ as maps to $\R^n\setminus\{0\}, $ i.e., the homotopy avoids zero.} see Lemma~\ref{l:perturb-ext}.
Consequently, we get that all the robust properties of $f^{-1}(0)$ are determined by the homotopy 
class of $f|_A$---a much more coarse and robust descriptor than the original map $f$.

{\def\thetheorem{A}
\addtocounter{theorem}{-1} 
\begin{theorem}
\label{t:hom-class} 
Let $X$ be a compact Hausdorff space, and $r>0$ be fixed. 
Then 
\begin{enumerate}
\item[(1)] The family $Z_{< r}(f)$ is determined by $A:=\{x: |f(x)|\geq r\}$ and the homotopy class of $\bar f\:A\to S^{n-1}$ 
defined by $\bar f(x):=f(x)/|f(x)|$.
\item[(2)] If  the pair $A\subseteq X$ can be triangulated and $\dim X\leq 2n-3$, then  $Z_{<r}(f)$
is determined by $A$ and the homotopy class of the quotient $f_{/A}: X/A\to S^n\simeq \R^n/\{x: |x|\geq r\}$ induced by the map of pairs 
$f: (X,A)\to (\R^n, \{x: |x|\geq r\})$.
\end{enumerate}
Once the space $B:=\{x: |f(x)|=r\}$ is given\footnote{In a ``generic'' case, $B=\partial A$  so then it is already encoded in $A$ in some sense.} in addition to the information above, then also $Z_{\le r}(f):=\{g^{-1}(0)\mid g\:X\to\R^n
\text{ such that }\|f-g\|\le r\}$ is determined by the homotopy classes specified in (1) or (2).
\end{theorem}}

The map $X/A\to S^n$ defined in part (2) will be denoted by $f_{/A}$ further on. 
For the set of all homotopy classes of maps from $X$ to $Y$ we will use the standard notation $[X,Y]$.
Part (2) strengthens the part (1), because the homotopy class of $f_{/A}$ is always determined by the homotopy class of $\bar f$
but not vice versa. In the dimension range $m:=\dim X\leq 2n-4$ the sets $[X,S^{n-1}]$, $[A,S^{n-1}]$ and $[X/A, S^n]$ 
possess an Abelian group structure and are called \emph{cohomotopy groups}. Then there is a sequence of homomorphisms 
\begin{equation}
\label{e:les1}
[X,S^{n-1}]\stackrel{i^*} {\longrightarrow} [A,S^{n-1}]\stackrel{\delta}{\longrightarrow}[X/A,S^n]
\end{equation}
where $i^*$ is induced by restriction and $\delta$ maps $[\bar f]$ to $[f_{/A}]$. 
Moreover, the sequence is exact, that is, $\ker \delta=\mathrm{Im}(i^*)$. So $[f_{/A}]$ only determines a coset
$[\bar f]+\mathrm{Im}(i^*)$ in $[A,S^{n-1}]$. The case $m=2n-3$ is more subtle but still $[f_{/A}]=\delta [\bar f]$ and it determines
$Z_{<r}(f)$ completely.
The bound $m\leq 2n-3$ from Theorem~\ref{t:hom-class} (2) is sharp.\footnote{Part (2) of the theorem may fail for $m=2n-2$. 
Let $n=6$ and $m=10$, $X$ be a unit ball in $\R^{10}$, $A=S^9$ and $f\: X\to \R^6$ be defined by 
$f(x)=|x|\eta(x/|x|)$ where $\eta\in [S^9, S^5]$ is a nontrivial element. 
Each $1$-perturbation of $f$ has a root in $B^{10}$ but this information is lost in $[X/A,S^{6}]\simeq [S^{10}, S^6]\simeq \{0\}$.}

\heading{Persistence of robust properties of zero sets.} 
We would like to understand the families $Z_{< r}(f)$ 
not only for one particular $r$ but for all robustness radia $r>0$ simultaneously.
The proper tool to describe it is the concept of \emph{persistence modules}.

We define a \emph{pointed Abelian group} to be a pair $(\gr,a)$ where $\gr$ is an Abelian group and $a\in \pi$ is its distinguished element.
A homomorphism of pointed groups $(\gr,a)\to (\gr',a')$ is a homomorphism $\gr\to \gr'$ that maps $a$ to $a'$. Under this definition, 
pointed Abelian groups naturally form a category. 
We define a \emph{pointed persistence module} to be a functor from $\R^+$ (considered as a poset category) to the category of pointed 
Abelian groups, explicitly $((\pi_r, a_r)_r, (\varphi_{s,r})_{0<r\leq s})$ where $\varphi_{s,r}: (\pi_r, a_r)\to (\pi_s, a_s)$ is
a homomorphism of pointed Abelian groups and $\varphi_{t,s}\varphi_{s,r}=\varphi_{t,r}$ for any $0<r\leq s\leq t$.
We define the \emph{interleaving distance} between two pointed persistence modules $\Pi$ and $\Pi'$ in the usual way 
as the infimum over all $\delta$ such that there exist families
of morphisms $u_r: (\gr_r, a_r)\to (\gr_{r+\delta}', a_{r+\delta}')$ and 
$v_r: (\gr_r', a_r')\to (\gr_{r+\delta}, a_{r+\delta})$ such that 
$v_{r+\delta} u_r=\varphi_{r+2\delta, r}$ and $u_{r+\delta} v_r=\varphi_{r+2\delta, r}'$ for all $r>0$~\cite{Chazal:2009,Chazal:2012}.


We use the pointed cohomotopy groups naturally coming from Theorem~\ref{t:hom-class} (2) as there is less redundant information than in part (1)
and the condition $\dim X\leq 2n-3$ will be needed for our computability results anyway.
For $r\leq s$, let $A_r$ and $A_s$ be $\{x: |f(x)|\geq r\}$ and $\{x: |f(x)|\geq s\}$, respectively. 
We define a subgroup $\gr_r$ of $[X/A_r,S^n]$ by 
\begin{equation}
\label{e:pi_r}
\gr_r:=\big\{[g_{/A_r}]\in [X/A_r, S^n]\,\,|\,\, g: (X,A_r) \to (\R^n, \{x:\, |x|\geq r\})\big\},
\end{equation}
that is, in the language of the sequence of homomorphisms (\ref{e:les1}), $\gr_r=\im \delta$.
 The quotient map $X/A_s\to X/A_r$ induces a natural map
$\varphi_{s,r}: \gr_r\to \gr_s$ that takes $a_r:=[f_{/A_r}]$ to $a_s:=[f_{/A_s}]$. 
Each quotient map $X/A_t\to X/A_r$ factorizes into quotient maps through $X/A_s$ for every $r\le s\le t$ 
and thus the homomorphisms $\varphi_{s,r}$ behave as required. Therefore the collections $\big(\gr_r,[f_{/A_r}]\big)_{r>0}$ and 
\(\big(\varphi_{s,r}\big)_{s\ge r\ge 0}\) form a pointed persistence module that we will denote by $\Pi_f$ and referred to as 
\emph{cohomotopy persistence module}. 

A simple observation is that the assignment $f\mapsto \Pi_f$ is stable with respect to the interleaving distance $d_I$: more precisely,
it satisfies $d_I(\Pi_f, \Pi_g)\leq \|f-g\|$.
It even holds that the interleaving distance is bounded by the so-called \emph{natural pseudo-distance} $d_N(f,g)$ between $f$ and $g$, that is, the infimum of $\|f-gh\|$ over all self-homeomorphisms $h\:X\to X$ (compare \cite{Cerri:2013a}).

If $\F$ is a field, then $\Pi_f\otimes \F$ is a pointed persistence module consisting of pointed vector spaces that are 
pointwise finite-dimensional. The distinguished elements $([f_{/A_r}]\otimes 1)_r$ generate a direct summand and the canonical decomposition
of $\Pi_f\otimes \F$ into interval submodules~\cite{Boevey:2012,Carlsson:2005} yields a 
\emph{pointed barcode}: this is a multiset of intervals with at most one distinguished interval. The distinguished interval 
corresponds to the distinguished direct summand whenever it is nontrivial.
The usual notion of \emph{bottleneck distance} easily generalizes to pointed barcodes: it also holds that the bottleneck distance between
$\Pi_f\otimes \F$ and $\Pi_g\otimes \F$ is bounded by $\|f-g\|$. Formal definitions and 
proofs containing justifications of these remarks are included in Section~\ref{s:modules}.

{\def\thetheorem{B}
\addtocounter{theorem}{-1} 
\begin{theorem}[Computability]
\label{t:computability}
Let $X$ be an $m$-dimensional simplicial complex, $f\: X\to\R^n$ be simplexwise linear with rational values on the vertices and $m\leq 2n-3$. For each $r>0$ let $A_r:=\{x\:|f(x)|\ge r\}$ where $|\cdot|$ denotes $\ell_1, \ell_2$ or $\ell_\infty$ norm.  

\begin{itemize}
\item[(1)]
The isomorphism type of the cohomotopy persistence module 
$$\Pi_f=\left(\big(\gr_r, [f_{/A_r}]\big)_{r>0},\,\big(\varphi_{s,r}\big)_{s\ge r\ge 0}\right)$$ can be computed. If $n$ is fixed,
the running time is polynomial with respect to the size of the input data representing $f\: X\to\R^n$.

\item[(2)]
If $\F$ is $\Q$ or a finite field and $n$ is fixed, then the pointed persistence barcode associated with $\Pi_f\otimes \F$ can be computed in polynomial time.
\end{itemize}
\end{theorem}}
Remarks on the theorem follow:
\begin{itemize} 
\item In the setting of the theorem, $\varphi_{s,r}$ is an isomorphism whenever $[r,s)$ contains none of so-called 
\emph{critical values of $f$}. There are only finitely many critical values and the isomorphism type of the persistence module is 
determined by a tuple of critical values \(s_1,\ldots,s_k\) of $f$, a sequence
$$
\gr_{r_0} \stackrel{\varphi_{r_1, r_0}}{\longrightarrow}\cdots
\stackrel{\varphi_{r_k,r_{k-1}}}{\longrightarrow}
\gr_{r_k}
$$
for $r_0<s_1<r_1<\ldots <s_k<r_k$ and the ``initial'' homotopy class of $f_{/A}$ in $\pi_{r_0}$.
\item Under the assumptions of the theorem, for any simplicial subcomplex $Y$ of $X$ and $r>0$, the problem $Y\in Z_{<r}(f)$ 
is decidable.\footnote{This 
amounts to the extendability of $\bar g: A_r\to S^{n-1}$ to the closure of the complement of certain regular neighborhood of $Y$ 
whenever $Y\cap A_r=\emptyset$ and $[\bar g]\in\delta^{-1}[f_{/A_r}]$ is arbitrary.}
In the special but important case of $Y=\emptyset$, it is equivalent to the triviality of $[f_{/A_r}]$. 
Thus the ``robustness of the existence of zero'' equals to the minimal $s_r$ such that
$[f_{/A_r}]=0$, equivalently, the length of the distinguished bar in a~suitable barcode representation. 
\item If $h: X\to X$ is a homeomorphism and $r$ a rotation of $\R^n$, then $\Pi_f$ and $\Pi_{r\circ f\circ h}$ are isomorphic. 
From this viewpoint, the computable bottleneck distance between two barcode representations of $\Pi_f\otimes F$ and $\Pi_g\otimes F$
only measures ``essential'' differences between robust properties of $f^{-1}(0)$ and $g^{-1}(0)$.

\item If $\dim X>2n-3$, then we may still define a persistence structure via part (1) of 
Theorem~\ref{t:hom-class} using $([A_r, S^{n-1}], [\bar f])$ instead of $(\pi_r,[f_{/A}])$.
In some particular dimensions (such as $n=1,2,4$, see below) this structure can be computed. 
The interleaving distance can be defined in the usual way and $d_I(\Pi_f, \Pi_g)\leq \|f-g\|$
still holds. 
\item We also remark that homotopy of two given maps can be algorithmically tested in all dimensions~\cite{Filakovsky:2013}
which can be used to verify the equality $Z_{<r}(f)=Z_{<r}(g)$ in some cases.
\end{itemize}

\heading{Low dimensional cases.} 
If $m=\dim X<n$, then $Z_{<r}(f)$ contains $\emptyset$ and consequently all closed subsets of 
$X$ contained in $X\setminus A$, so there is not much to compute.
The condition $n\leq m\leq 2n-3$ is never satisfied for $n\leq 2$ but in these cases the element $[\bar f]\in [A,S^{n-1}]$ is computable
and we may use part (1) of Theorem A.

The case $n=1$ describes scalar valued functions. Then the homotopy class $\bar f: A_r\to S^0$ consists of a set of pairs 
$(A_r^j, s_r^j)$ where $A_r^1,\ldots, A_r^{n(r)}$ are the connected components of $A_r$ and $s_r^j$ is the sign of $f$ on $A_r^j$. 
If $r<s$, each $A_s^k$ is a subset of a unique $A_r^j$ and the sign is inherited. 
The structure of these components and signs can clearly be computed from the input such as in Theorem~\ref{t:computability}.

The case $n=2$ is also easy to handle. If $A_r$ is a simplicial complex of any dimension, $[A_r,S^1]$ is an Abelian group
naturally isomorphic to the cohomology group $H^1(A_r, \Z)$ \cite[II, Thm 7.1]{HuBook} which can easily be computed 
by standard methods~\cite{Edelsbrunner:2010}.
The inclusion $A_s\hookrightarrow A_r$ induces a homomorphism $[A_r, S^1]\to [A_s, S^1]$ and the whole persistence module 
consisting of these groups and homomorphisms is computable.

For $n=m=3$, the condition $m\leq 2n-3$ is satisfied.
However, if the input is a $4$-dimensional finite simplicial complex $X$ and 
a simplexwise linear map $f: X\to\R^3$, then we may only hope for partial and incomplete algorithmic results, because
$\emptyset\in Z_{<r}(f)$ is then an undecidable problem by~\cite{nondec}.

Surprisingly, $n=4$ is a special case because $[Y, S^3]$ is an Abelian group for any simplicial complex $Y$: the group operation
can be derived from the quaternionic multiplication in the unit sphere $S^3$. The cases $m=\dim Y\leq 5$ are covered in our theorems above and the computability of $[Y, S^3]$ for higher-dimensional $Y$ is a work in progress.

\heading{Additional information contained in $[f_{/A}]$.}
\begin{figure}
  \begin{center}
    \includegraphics[scale=1.1]{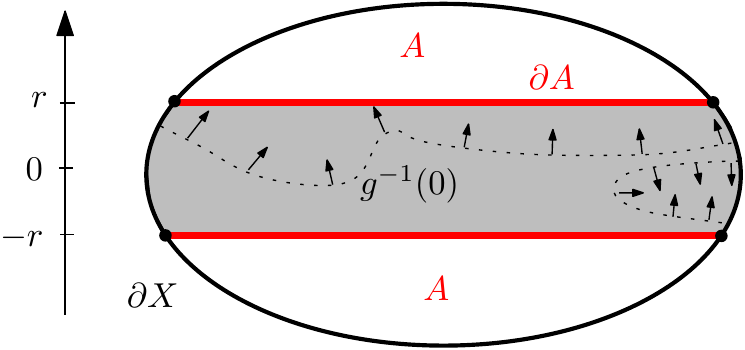}
\caption{Smooth zero sets with a framing. The map $f: X\to\R$ is a~projection to the vertical
dimension and the zero set of each regular $r$-perturbation $g$ of $f$ is a smooth submanifold of $X$ disjoint from $A=\{x: |f(x)|\geq r\}$.  
The vector field represents the gradient of $g$ on its zero set.}
\label{fig:intro_smooth}
  \end{center}
\end{figure} 
Theorem~\ref{t:hom-class} cannot be fully reversed. If $Z_{< r}(f)$ is given, then  $A={X\setminus \cup Z_{< r}(f)}$ can be reconstructed, 
but the corresponding element in $\pi_r\subseteq [X/A,S^{n}]$ is not uniquely determined.\footnote{If $f$ is the identity on a unit $n$-ball, 
we have $Z_{< r}(f)=Z_{< r}(-f)$ for each $r\in (0,1]$ but if $n$ is odd, then $[f_{/A}]\neq [(-f)_{/A}]$.}
In general, there is a~many-to-one correspondence between $\pi_r$ 
and the collection $\{Z_{< r}(f)\mid f\:X\to \R^n\text{ with } \{x: |f(x)|\geq r\}=A\}$:
the distinguished elements in $\pi_r$ still carry more information than is needed to determine $Z_{< r}(f)$.
A natural question is, how to understand this additional information and its geometric meaning? 

We will show that we can achieve a one-to-one correspondence between homotopy classes and zero sets
if we enrich the family of zero sets with an additional structure that carries a {directional information} associated to the zero sets. 
For any $x\in f^{-1}(0)$, this structure contains \emph{gradients} of the components of $f$ in $x$, 
see Fig.~\ref{fig:intro_smooth} for an illustration.

To formalize this, assume that $X$ is a smooth compact $m$-manifold, $f$ is smooth and $0$ is a \emph{regular value} of $f$:
that is, the differential $df(x)$ has (maximal) rank $n$ for each $x\in f^{-1}(0)$.
This implies that $f^{-1}(0)$ is an $m-n$ dimensional submanifold of $X$. 
Assume further that $0$ is also a regular value of $f|_{\partial X}$.
We will call such functions $f$ \emph{regular}: these properties are by no means special but rather
generic by Sard's theorem~\cite{Milnor:97}.
A regular function $g: X\to\R^n$ such that $\|g-f\|<r$ will be called a \emph{regular $r$-perturbation} of $f$.
Now we are ready to define the enriched version of the family of zero sets
\begin{equation*}
\label{def:Z_r^fr}
\ZrFr(f):=\big\{\big(g^{-1}(0), dg|_{g^{-1}(0)}\big):\,g\,\text{  is a~regular $r$-perturbation of $f$}\big\}.
\end{equation*}
Each element of $\ZrFr(f)$ carries the information about the zero set of some $g$ and the differential $dg$ at this zero set.
The submanifold $g^{-1}(0)$ together with $dg|_{g^{-1}(0)}$ is called a \emph{framed submanifold} and can be geometrically
represented via $n$ gradient vector fields on $g^{-1}(0)$ such as in Figure~\ref{fig:intro_smooth}.

Two framed $k$-submanifolds $N_1$ and $N_2$ of $X$ are \emph{framed cobordant},
if there exists a framed $(k+1)$-dimensional submanifold $C$ of $X\times [0,1]$ such that 
$C\cap (X\times \{0\})=N_1\times \{0\}$, $C\cap (X\times \{1\})=N_2\times \{1\}$ and the framing of $C$ in $X\times \{0,1\}$ 
is mapped to the framing of $N_1, N_2$ via the canonical projection $X\times [0,1]\to X$.
The manifold $C$ is called a \emph{framed cobordism}:  see Fig.~\ref{fig:cob} for an illustration 
and Section~\ref{a:cobordism2perturbation} for a~precise definition in case when $X$ is a manifold with boundary. 
\begin{figure}
  \begin{center}
    \includegraphics[scale=1.1]{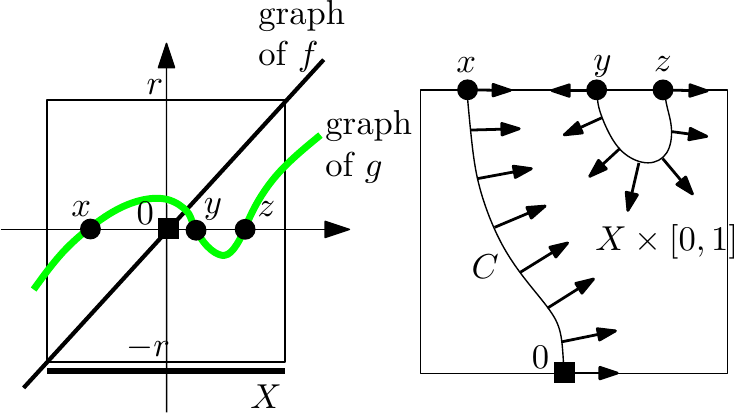}
\caption{The framed zero set of a perturbation $g$ of $f$ consists of three points $x,y$ and $z$ and a framing indicating the directions
in which $g$ is increasing. It is framed cobordant to the framed zero set $(0,\to)$ of $f$.}
\label{fig:cob}
  \end{center}
\end{figure} 
{\def\thetheorem{C}
\addtocounter{theorem}{-1} 
\begin{theorem} 
\label{t:cobordism2perturbation}
Assume that $X$ is a smooth compact $m$-manifold, $r>0$, $A\subseteq X$ is closed,
$m\leq 2n-3$, and $\pi_r$ be the subgroup of $[X/A, S^n]$ defined by (\ref{e:pi_r}).  Then there is a bijection
$$
\big\{\ZrFr(f)\mid f: X\to\R^n\text{ such that }A=\{x: |f(x)|\geq r\}\big\}\longleftrightarrow \gr_r
$$ 
satisfying that each $\ZrFr(f)$ is mapped to $[f_{/A}]$.
Moreover, each $\ZrFr(f)$ is a~\emph{framed cobordism class} of framed $(m-n)$-submanifolds disjoint from $A$.
Here the framed cobordisms are also required to be disjoint from $A\times [0,1]$.
\end{theorem}}
If $A$ is given, then any framed zero set determines its framed cobordism class $\ZrFr(f)$ and hence
$[f_{/A}]$. It follows that $[f_{/A}]$ is a property common to all elements of $\ZrFr(f)$, that is, an invariant of $\ZrFr(f)$.
This invariant is complete, as it determines all of $\ZrFr(f)$.

In its special case, Theorem~\ref{t:cobordism2perturbation} claims that 
whenever $f: X\to\R^n$ is such that $[f_{/A}]=0$, then $\ZrFr(f)$ consists exactly of all framed $(m-n)$-submanifolds that are 
framed null-cobordant in $X\setminus A$. This particular claim can also be derived from~\cite[Theorem 3.1]{Koschorke}.

If $m\leq 2n-3$ is violated, then the framed zero sets of regular perturbations are still framed cobordant but $\ZrFr(f)$ 
is only a~subset of the full framed cobordism class. It is an interesting question for further research to find
the additional invariants of framed zero sets in these cases.
%

\heading{Related work.}
One of the roots of our research comes from zero verification.
If $X$ is a product of intervals and $f: X\to\R^n$ is defined in terms of interval arithmetic,\footnote{
That is, there is an algorithm that computes a superset of $f(X')$ for any subbox $X'\subseteq X$ with rational vertices.}
then the nonexistence of zeros of $f$ can often be verified by interval arithmetic alone \cite{Neumaier:90}.
However, the proof of existence requires additional ingredients such as Brouwer fixed point theorem \cite{Rump:2010}
or topological degree computation \cite{Collins:08b,Franek_Ratschan:2015}.
These techniques are applicable for domains $X$ of dimension $n$ and succeed only if the zero is $r$-robust for some $r>0$. 
Naive applications of these techniques fail in the case of ``underdetermined systems'' $f(x)=0$ 
where the dimension of the domain $X$ of $f$ is larger than $n$. 
In \cite{nondec} we analyzed the problem of existence of an $r$-robust zero of functions $f: X\to\R^n$ where $X$ 
is a simplicial complex of arbitrary dimension.

Another parallel line of related research is the field of persistent homology which analyzes 
properties of scalar functions (rather than their zero sets) via persistence modules 
build up from the homologies of their sublevel sets $f^{-1}(-\infty, r]$ for all $r\in \R$. 
Persistent homology has been generalized to the case of $\R^n$-valued functions
\cite{Carlsson:2009_pers,Cerri:2013a,Cerri:2013b,Cagliari:2010}. 

\heading{\it{Well groups.}} 
Well groups associated to $f : X \to Y$ and a subspace $Y' \subseteq Y$ 
describe homological properties of the preimage $f^{-1}(Y')$ which persist if we perturb the input function $f$.
We include a formal definition for the case of $Y=\R^n$ and $Y'=\{0\}$. Let $W$ be the space of potential zeros of all $r$-perturbations, that is, $\{x: |f(x)|\leq r\}$. Then the well groups $U_*(f,r)$ are subgroups of homology groups $H_*(W)$ consisting of classes \emph{supported} by the zero set of each $r$-perturbation $g$ of $f$. Formally, 
$$U_*(f,r):=\bigcap_{Z\in Z_{\le r}(f)} \im\big(H_*(Z\hookrightarrow W)\big)
$$
where $Z\hookrightarrow W$ is the inclusion and $H_*$ is a convenient homology theory.
Most notably, $U_0(f,r)=0$ whenever $f$ has no $r$-robust zero, i.e., \(\emptyset\in Z_{\le r}(f)\) and therefore 
the same undecidability result \cite{nondec} applies to well groups. Obviously, each well group is a property of $Z_{\le r}(f)$ 
and is therefore ``encoded'' in the homotopy class of $\bar f$.
However, the decoding seems to be a difficult problem, see \cite{FK:well-dcg} for some partial results and \cite{interlevel,chazal} for previous algorithms for special cases $n=1$ and  $\dim X=n$.

Well groups for various radia $r$ fit into a certain
zig-zag sequence that yields  so-called \emph{well diagrams}---a multi-scale version of well groups that is provably stable under perturbations of $f$ \cite{well-group}.

Summarizing our opinion, well diagrams  provide very general tool for robust analysis that uses accessible and geometrically intuitive language of homology theory. In addition, they present a challenging computational problem deeply interconnected with homotopy theory. However, their
computability status is worse  than that of cohomotopy groups and they fail to capture some properties of $Z_ {\le r}(f)$ \cite{FK:well-dcg}.

%

\section{Illustrating examples.} 
\heading{Intermediate value theorem.}
In the motivating example from Figure~\ref{fig:fig-intro}, the family $Z_{<r}(f)$ is characterized by the map $A\to S^0=\{+,-\}$ such that
each component of $A$ is mapped to a different element of $S^0$. 
By the intermediate value theorem, any curve connecting the two components of $A$ intersects the zero set of any $r$-perturbation of $f$.  
The set $Z_{<r}(f)$ is determined by the element of $[A,S^0]\simeq [S^0, S^0]$, illustrating Theorem A (1).

There are two non-constant elements of $[A,S^0]$, represented by $f$ and $-f$. 
They give rise to identical sets $Z_{<r}(f)=Z_{<r}(-f)$. 
However, the framed version $\ZrFr(f)$ and $\ZrFr(-f)$ are different, as the gradient information encodes on which side of the zero set
is the function positive, resp. negative.

\heading{Topological degree.}
Consider functions $\R^n\to \R^n$ and $r>0$ such that $W:=\{x:\, |f(x)|< r\}$ is a topological $n$-disc.
In this case, $Z_{< r}(f)$ is determined by the degree of $$\partial W\simeq S^{n-1} \stackrel{f}{\longrightarrow} \R^n\setminus\{0\}\simeq S^{n-1}.$$
If this degree is nonzero, then $f$ is not extendable to all of $W$ and each $r$-perturbation of $f$ has a root in $W$.
It is not hard to show that $Z_{<r}(f)$ then consists of all \emph{non-empty} closed sets contained in $W$. On the other hand,
if the degree is zero, then some $r$-perturbation of $f$ avoids zero and $Z_{<r}(f)$ consists of \emph{all} closed sets contained in $W$.
The degree is clearly determined by the homotopy class of the map $[\bar f]\in [A_r, S^{n-1}]$ where $A_r=\{x:\,\,|f(x)|\geq r\}$.

While $Z_{<r}(f)$ does not distinguish various nonzero degrees, the refined version 
$\ZrFr(f)$ from Theorem C does.\footnote{The degree determines and is determined by the image of $[A_r, S^{n-1}]$ in
$[X/A_r, S^n]\simeq [S^n, S^n]\simeq \Z$.}
If the degree is $k\in\Z$, then $\ZrFr(f)$ consists of all finite framed point sets in $W$ such that the 
difference between positively and negatively oriented points is exactly $k$. 
Thus not only does $\ZrFr(f)$ determine the degree, but so does \emph{each element} of $\ZrFr(f)$. 

\heading{Higher order obstructions.}
The following example, taken from~\cite{FK:well-dcg}, illustrates the strength of Theorem A in a~situation where \emph{well groups} (based 
on homology theory) are not sufficient to describe $Z_{<r}(f)$.
Let $X=S^2\times B^4$ where $S^2$ is the standard unit sphere, and $f: X\to\R^3$ is defined by $f(x,y)=|y|\,\eta(y/|y|)$ 
where $\eta: S^3\to S^2$ is the Hopf fibration.

The Hopf map $\eta$ can not be extended to $B^4\to S^2$, and so each $1$-perturbation of $f$ has a~root in each section $\{x\}\times B^4$.
In particular, the zero set of a~perturbation cannot be discrete.

Consider another map $g: S^2\times B^4\to \R^3$ defined by $g(x,y):=|y| \varphi(x,y/|y|)$ where $\varphi: S^2\times S^3\to S^2$
is defined as the composition $S^2\times S^3\stackrel{\wedge}{\to} (S^2\times S^3)/(S^2\vee S^3)\simeq S^5\stackrel{\nu}{\to} S^2$
where $\nu$ is a~homotopically nontrivial map. In this case, we showed in~\cite{FK:well-dcg} that every $1$-perturbation of $g$ has a~zero
but it may be a~singleton.  
Thus $Z_{<r}(f)\neq Z_{<r}(g)$ and the map $(x,y)\mapsto \eta(y)$ is not homotopic to $(x,y)\to \varphi(x,y)$ 
as maps from $A:=S^2\times S^3$ to $S^2.$
Note that the sphere-valued map $g|_{A}$ is extendable to the $5$-skeleton of $X$, while $f|_A$ is only extendable to the $3$-skeleton of $X$. 

However, $f$ and $g$ give rise to isomorphic well groups which are both zero in all positive dimensions.
While the zero set of $f$ is the two-sphere $S^2\times \{0\}$, there exist arbitrarily small perturbations of $f$ having 
the zero set homeomorphic to $S^3$, killing a~potential nontrivial element of the second homology of the zero sets of perturbations.

Less technically, the information that ``the zero set of each $r$-perturbation of $f$ intersects each section $\{x\}\times B^4$''
is lost in the well group description of $Z_{<r}(f)$.

\heading{Cohomotopy barcode.} 
The example from the previous heading immediately generalizes to $f,g: S^2\times B^{n+1}\to \R^n$ 
defined via nontrivial elements 
$$\eta\in \pi_{n}(S^{n-1})\simeq\Z_2\quad \text{ and }\quad \nu\in \pi_{n+2}(S^{n-1})\simeq \Z_{24}$$
for large enough $n$. It was shown in~\cite{FK:well-dcg} that the well groups and well modules associated to $f,g$ are 
trivial in all positive dimensions, although $Z_{<r}(f)\neq Z_{<r}(g)$ for $r\in (0,1]$.

For each $r\in (0,1]$, the exact cohomotopy sequence 
$$0=[X,S^{n-1}]\to [A_r, S^{n-1}]\stackrel{\delta}{\to} [X/A_r,S^n]\to [X,S^n]=0$$
shows that we have an~isomorphism $$\pi_r=\mathrm{Im}(\delta)=[X/A_r,S^{n-1}]\simeq [A_r,S^{n-1}]\simeq [S^2\times S^{n}, S^{n-1}].$$
and $\varphi_{r,s}: \pi_r\to \pi_s$ is the identity (under the above identification). 
The group $[S^2\times S^{n}, S^{n-1}]$ equals\footnote{It follows from the exact sequence presented in~\cite[Problem 18.31]{Skopenkov}
and the relatively easy facts that the second arrow of this sequence is a~surjection and the sequence splits.}
$\pi_{n}(S^{n-1})\times \pi_{n+2}(S^{n-1})\simeq \Z_2\times \Z_{24}$
and tensoring with the field $\Z_3$ would yield diagrams with one bar only.
This bar would be distinguished in the diagram of $g$ but not so in the diagram of $f$. Thus a~pointed 
cohomotopy barcode can distunguish two functions $f$ and $g$ with equal well groups and well modules.

\section{Proof of Theorem A (strict case)}
\label{s:hom-class}
In this Section we will give a proof of Theorem A for the case of strict-inequalities and refer the non-strict case $Z_{\leq r}(f)$, 
which is more technical, to the Appendix (page~\pageref{s:rest_thm_A}).

The proof utilizes certain properties of compact Hausdorff spaces. 
We say that a pair of spaces $(Y,Z)$ satisfies the \emph{homotopy extension
property with respect to a space $T$} whenever each map $H'\:Y\times\{0\}\cup
Z\times[0,1]\to T$ can be extended to $H\:Y\times[0,1]\to T$.
The map $H'$ as above will be called a \emph{partial
homotopy} of  $H'|_{Y}$ on $Z$. It follows from \cite[Prop. I.9.3]{HuBook}
that, once $K$ is compact Hausdorff and $T$ triangulable, every
pair of closed subsets $(Y,Z)$ of $K$ satisfies the homotopy
extension property with respect to $T$. 

In addition, for every two disjoint closed subsets $V$ and $W$
in a compact Hausdorff space $X$ there is a \emph{separating
function} $\chi\:X\to[0,1]$. That means, there is a function
$\chi\:X\to[0,1]$ that is $0$ on $V$ and $1$ on $W$. It is easily
seen that the values $0$ and $1$ above can be replaced by arbitrary
real values $s<t$.

Finally, every homotopy $H\:Y\times[0,1]\to T$ of the form $H(y,t)=H(y,0)$
will be called \emph{stationary}.

%
\begin{lemma}[From perturbations to homotopy perturbations]\label{l:perturb-ext}
Let $f\:X\to \R^n$ be a map on a compact Hausdorff space $X$
and let $A:=\{x: |f(x)|\geq r|\}$. Then the families 
\begin{equation}
\label{e:array_1}
\tag{A}
\{g^{-1}(0)\mid g\text{ is a strict $r$-perturbation of } f\}\text{,}
\end{equation}
\begin{equation}
\label{e:array_2}
\tag{B}
\{h^{-1}(0)\mid
h\:(X,A)\to (\R^n,\R^n\setminus\{0\})\text{, }h|_A\sim f|_A\}
\text{ and}
\end{equation}
\begin{equation}
\label{e:array_3}
\tag{C}
\{e^{-1}(0)\mid e\:X\to\R^n \text{
is an extension
of } f|_A\} 
\end{equation}
are all equal. Moreover, if an extension $e: X\to\R^n$ of $f|_A$ is given, then the
strict $r$-perturbation $g$ of $f$ such that $g^{-1}(0)=e^{-1}(0)$ can be chosen to be a multiple of $e$ by a positive scalar function.
\end{lemma}

\begin{proof}We will prove that the sequence of inclusions (A)
$\subseteq$ (B) $\subseteq$ (C) $\subseteq$ (A)  holds. 
The additional relation between $e$ and $g$ will be shown in the construction of $g$ in the (C)$\subseteq $ (A) part.

\heading{(A) is a subset of (B):} Each strict $r$-perturbation $g$ of $f$ is nowhere zero on $A$ and the straight line homotopy
$F(a,t)=t\,g(a)+(1-t) f(a)$ satisfies $F(t,a)\neq 0$ for any $(a,t)\in A\times [0,1]$. Indeed, each line shorter than $r$ starting at a point at least $r$ away from zero has to avoid zero.


\heading{(B) is a subset of (C):} We start with a map of
pairs $h\:(X,A)\to(\R^n,\R^n\setminus\{0\})$ such that $h|_A$ is homotopic to $f|_A$ and want to construct an extension
$e$ of $f|_A$ such that $e^{-1}(0)=h^{-1}(0)$. To that end, let us choose
a value $\epsilon>0$ such that $\min_{x\in A}|h(x)|\ge 2\epsilon$
and let us define $Y:=|h|^{-1} [\epsilon,\infty)$. The
partial homotopy of $h$ on $|h|^{-1}(\epsilon)\cup A$
that is stationary on $|h|^{-1} (\epsilon)$ and equal
to the given homotopy $h|_A\sim f|_A$ on $A$ can be extended
to $H\:Y\times[0,1]\to\R^n\setminus\{0\}$ by the homotopy extension
property. The homotopy extension property holds because all the
considered maps take values in a triangulable space $\{x\in\R^n\:|x|\in[\epsilon, M]\}$ for some $M\in\R$. 

The desired extension $e$ can be defined to be equal to $h$ on
$|h|^{-1}[0,\epsilon]$ and equal to $H(\cdot,1)$ on $Y$.

\heading{(C) is a subset of (A):} We start with an extension
$e\:X\to\R^n$ of $f|_A$ and want to construct a strict $r$-perturbation
$g$ of $f$ such that $g^{-1}(0)=e^{-1}(0)$. 

The set $U:=\{x\in X\: |e(x)-f(x)|<r/2\}$ is an open neighborhood
of $A$.
Due to the compactness of $X$, there exists $\epsilon\in
(0,r/2)$ such that $|f|^{-1}[r-\epsilon, \infty)\subseteq U$ (otherwise,
there would exist a sequence $x_n\notin U$ with $|f(x_n)| \to
r$ and a convergent subsequence $x_{j_n}\to x_0$, where $x_0\in
A\subseteq U$, contradicting $x_{j_n}\notin U$).

Let $\chi\:X\to[\epsilon/(2 \|e\|),1]$ be a separating function
for $A$ and  $W:=|f|^{-1} [0,r-\epsilon]$, that is, a continuous
function that is $\epsilon/(2\|e\|)$ on $W$ and $1$ on $A$. The
map $g\:X\to\R^n$ defined by 
$$g(x):=\chi(x)e(x)$$
is a strict $r$-perturbation of $f$. Indeed, for $x\in A$, $f(x)=g(x)$ by definition and for $x\in W$, we have $|g(x)-f(x)|\leq
\epsilon/2 +(r-\epsilon)<r$. Otherwise, $x\in U\setminus A$ and
then$$|g(x)-f(x)|\leq \chi(x)\underbrace{|e(x)-f(x)|}_{\leq r/2}+
(1-\chi(x))\underbrace{|f(x)|}_{\leq r}<r.$$  
\end{proof}

\begin{proof}[Proof of Theorem~\ref{t:hom-class}, Part (1).]
This follows directly from the equality between $(A)$ and $(B)$ in Lemma \ref{l:perturb-ext}: clearly the definition of the family 
$(B)$ depends on the homotopy class of $f|_A\:A\to\R^n\setminus\{0\}$ only. 
This homotopy class is uniquely determined by the homotopy class of $\bar f\:A\to S^{n-1}$.
\end{proof}

\heading{Cohomotopy groups.}
For Part (2) of the Theorem, we need the Abelian group structure of $[A, S^{n-1}]$ and $[X, S^{n-1}]$
(see~\cite[Chapter 7]{HuBook} for more details).
Assume first that $m\leq 2n-4$ which will make the proof easier: we will comment on the special case $m=2n-3$ at the end.
If $A\subseteq X$ are simplicial complexes of dimension $\leq 2n-4$, then both $[A, S^{n-1}]$ and $[X, S^{n-1}]$ are Abelian groups
with the group operation defined as follows.
Let $f,f'$ be maps $A\to S^{n-1}$. The image of the cellular approximation $a$ of $(f,f')\:A\to S^{n-1}\times S^{n-1}$ 
misses the top $(2n-2)$-cell, hence $a(A) \subseteq S^{n-1}\vee S^{n-1}$. The sum $f\boxplus f'\:A\to S^{n-1}$ is defined 
as the composition $\nabla a$ where $\nabla\:S^{n-1}\vee S^{n-1}\to S^{n-1}$ is the folding map. 
The homotopy class $[f\boxplus f']$ is independent of the choice of representative $f$ of $[f]$ and $f'$ of $[f']$ 
and is independent of the choice of the cellular approximation $a$ as well. 
It induces a binary operation in $[A, S^{n-1}]$ which is associative and commutative, the element $[\const]$ 
is neutral with respect to this operation and the inverse element to $f$ is obtained by composing $f$ with a map 
$S^{n-1}\to S^{n-1}$ of degree $-1$ that will be denoted by $\boxminus f$. 
The inclusion $i: A\hookrightarrow X$ induces a homomorphism $i^*: [X, S^{n-1}]\to [A, S^{n-1}]$ whose image
is a subgroup of $[A, S^{n-1}]$ that consists of homotopy classes of maps that are extendable
to $X\to S^{n-1}$. In particular, once $f\:A\to S^{n-1}$ is extendable to $X\supseteq A$, then $\boxminus f$ is as well. 
By \cite[Chapter VII]{HuBook}, there is an exact sequence of cohomotopy groups 
\begin{equation}
\label{e:exact}
[X, S^{n-1}]\stackrel{i^*}{\to} [A, S^{n-1}]\stackrel{\delta}{\to} [(X,A), (S^n, *)]\simeq [X/A, S^n]
\end{equation}
where $\delta$ maps the homotopy class $[\bar f]\in [A, S^{n-1}]$ to $[f_{/A}]$ defined in Theorem \ref{t:hom-class}.
The exactness of this sequence\footnote{That is, $\ker\delta=\mathrm{Im}(i^*)$.} implies that the $\delta$-preimage of 
$[f_{/A}]$ is $[\bar f]+\mathrm{Im} (i^*)$. To prove our statement, we need to show that this coset in $[A, S^{n-1}]$
uniquely determines $Z_{<r}(f)$. 
%
For maps $f_1,f_2\:X\to\R^n$ by $f_1\boxplus f_2$ we will denote an arbitrary extension $X\to\R^n$ of a representative 
$\bar f_1 \boxplus \bar f_2$ of $[\bar f_1]\boxplus [\bar f_2]$. 
By Theorem~\ref{t:hom-class} (1) the family $Z_{< r}(f_1\boxplus f_2)$ is independent of the choices of the representative and of the extension.

\begin{lemma}
\label{l:stronger} 
Let $A\subseteq X$ be cell complexes of dimension at most $2n-4$ and $f_1, f_2: X\to\R^n$
be such that $A=|f_1|^{-1}[r,\infty)=|f_2|^{-1}[r,\infty)$. Then
$$
Z_{< r}(f_1\boxplus f_2)\supseteq \{Z_1\cup Z_2\:Z_1\in Z_{< r}(f_1)\text{ and }Z_2\in Z_{< r}(f_2)\}.
$$
\end{lemma}
%
\begin{proof}[Proof of Lemma~\ref{l:stronger}] Let $g_i$ be a strict $r$-perturbation of
$f_i$ and $Z_i=g^{-1}(0)$ for $i=1,2$. We want to find a strict $r$-perturbation $g$ of $f_1\boxplus f_2$ with $g^{-1}(0)=Z:=Z_1\cup Z_2$.
Let as represent the functions $g_i$ in polar coordinates as $g_i=\bar g_i \, |g_i|$ where $\bar g_i: X\setminus Z_i\to S^{n-1}$ 
is defined by $g_i/|g_i|$. The map $g$ will be constructed in polar coordinates as $g=\bar g\,n$ for $\bar g: X\setminus Z\to S^{n-1}$
and $n: X\to\R_0^+$ such that $n$ will be zero on $Z$. The map $\bar g$ will be essentially $\bar g_1\boxplus \bar g_2$.
The only issue is that the definition of $\boxplus$ requires the domain to be  a cell complex 
(because it uses a cellular approximation of the map $(\bar g_1,\bar g_2)$) which $Y:=X\setminus Z$ is not. 
Thus we will need a sequence of cell complexes $Y_0\subseteq Y_1\subseteq\ldots$ contained in $Y$ such that $\bigcup Y_i=Y$. 
Let $\dist$ be a metrization of $X$ 
and for each $i=1,2,\ldots$ let $d_i\:X \to\R$ be PL functions less than $2^{-i-2}$ far from $\dist(Z, \cdot)$ in the max-norm. 
By PL we mean that each  $d_i$ is simplexwise linear on some triangulation of $X$. 
Let $O_i:=d_i^{-1}[0,2^{-i})$ and $Y_i:=X\setminus O_i$. 
We have that $Y_i\supseteq Y_{i-1}$ as sets and after a possible subdivision of these cell complexes we may assume that
$Y_{i-1}$ is a subcomplex of $Y_i$. 
Let $a_i\:Y_i\to S^{n-1}\times S^{n-1}$ be a cellular approximation of $((\bar g_1)|_{Y_i}, (\bar g_2)_{Y_i})$ that extends $a_{i-1}$ if $i>1$. 
Define $\bar h:=\bigcup_i \nabla a_i$ and $n:=\dist(Z,\cdot)$. Then $Z$ is the zero set of $h=\bar h\,n$ and the restriction of $\bar h$ to $A$ equals $(\bar g_1)|_A\boxplus (\bar g_2)|_A$.
Under the assumption $\dim A\leq 2n-4$, $(\bar g_1)|_A\boxplus (\bar g_2)|_A$ is well defined up to homotopy, is homotopic to 
$\bar f_1\boxplus \bar f_2$ and it follows from Lemma \ref{l:perturb-ext} that $Z\in Z_{<r}(f_1\boxplus f_2)$. 
 \end{proof}

\begin{proof}[Proof of Theorem A, Part (2)]
Assume first that $\dim X\leq 2n-4$.
For $\bar f_2\:A\to S^{n-1}$  extendable to $X\to S^{n-1}$ (i.e., $\emptyset\in Z_{< r}(f_2)$), 
we obtain $Z_{< r}(f_1\boxplus f_2)\supseteq Z_{< r}(f_1)$ by Lemma~\ref{l:stronger}. 
Consequently, since $\boxminus\bar f_2$ is also extendable,
$$
Z_{<r}(f_1) = Z_{< r}\big((f_1\boxplus f_2)\boxplus(\boxminus f_2)\big) \supseteq Z_{< r}(f_1\boxplus f_2).
$$
Hence $Z_{<r}(f_1)=Z_{<r}(f_1\boxplus f_2)$ for any $f_2$ such that $\bar f_2$ is extendable to a map $X\to S^{n-1}$. 
It follows that $Z_{<r}(f)$ only depends on the coset $[\bar f]+\im i^*$ in $[A, S^{n-1}]$.

Finally, we discuss the special case $m=2n-3$ that goes along the same lines with the following differences.
We replace $A$ by $A':=\partial A$, $X$ by $X':=\overline{X\setminus A}$ and $f$ by $f':=f|_{X'}$. 
Clearly \(\Zrf=Z_{<r}(f')\). The space $X'$ is still at most $2n-3$ dimensional but $A'$ is  at most $2n-4$ dimensional.
Instead of (\ref{e:exact}) we consider the sequence
$$
[X',S^{n-1}]\stackrel{i'^*}{\to} [A',S^{n-1}]\stackrel{\delta'}{\to} [X'/A',S^{n}]\stackrel{\iota^*}{\leftarrow}[X/A,S^n]
$$
which are all Abelian groups possibly except $[X', S^{n-1}]$ which is only a set. The map $\iota^*$ is induced by the inclusion $\iota\:(X',A')\to(X,A)$ and is an isomorphism by excision \cite[Chapter VII, Theorem 3.2]{HuBook}.
By \cite[Chapter VII, Lemma 9.1]{HuBook} this sequence is still exact at $[A',S^{n-1}]$,
that is, $\ker\delta'=\mathrm{Im}(i'^*)$. In particular, this implies that $\mathrm{Im}(i'^*)$ is a subgroup of $[A', S^{-1}]$ 
and $\delta'$ maps the quotient $[A',S^{n-1}]/\mathrm{Im}(i'^*)$ isomorphically onto $\mathrm{Im}(\delta')$, so that the preimage of
$[f'_{/A'}]$ is $[\bar f']+\mathrm{Im}(i'^*)$ which determines  $Z_{<r}(f')$ as above. It remains to check is that $\iota^*([f_{/A}])=[f'_{/A'}]$. This follows from the naturality of the exact sequence~\eqref{e:exact}, in particular, the commutativity of the square
$$
\xymatrix{
[A,S^{n-1}]\ar[d]^{\iota_{|A'}^*}\ar[r]^\delta&[X/A,S^n]\ar[d]^{\iota^*}\\
[A',S^{n-1}]\ar[r]^{\delta'}&[X'/A',S^{n}]
}
$$ as in \cite[Chapter VII, Proposition 4.1]{HuBook}, and from observing that 
$$\iota_{|A'}^*[\bar f] = [\bar f'].$$ 
\end{proof}

\section{Cohomotopy persistence modules.}\label{a:comput}
\label{s:modules}
\heading{Stability of cohomotopy persistence modules.} 
Let 
$$\Pi=\big(\varphi_{s,r}\:(\pi_r,a_r)\to(\pi_s,a_s)\big)_{s\geq r>0} \quad\text{ and}\quad  
\Pi'=\big(\varphi_{s,r}'\:(\pi_r',a_r')\to(\pi_s',a_s')\big)_{s\geq r>0}
$$
be two pointed persistence modules.
We define their interleaving distance $d_I(\Pi, \Pi')$ as the infimum over all $\delta>0$
such that there exists a family of homomorphisms $u_r: (\pi_r, a_r)\to (\pi_{r+\delta}', a_{r+\delta}')$ and
$v_r: (\pi_r', a_r')\to (\pi_{r+\delta}, a_{r+\delta})$ such that $v_{r+\delta} u_r=\varphi_{r+2\delta, r}$ and
$u_{r+\delta} v_r=\varphi_{r+2\delta, r}'$ holds for all $r>0$.


The first observation on cohomotopy persistence modules is that the assignment $f\mapsto \Pi_f$ is stable with respect to perturbations of $f$, 
namely, the interleaving distance of $\Pi_f$ and $\Pi_{f'}$ is bounded by \(\|f-f'\|\). 
Let $A_{r}=\{x: |f(x)|\geq r\}$ and $A_r'=\{x: |f'(x)|\geq r\}$ for all $r$ and assume that $\|f-f'\|<\delta$ for some $\delta>0$. 
This immediately implies $A_{r+\delta}'\subseteq A_{r}$. The straight line homotopy between
$f$ and $f'$ is nowhere zero on $A_{r+\delta}'$ and it induces a homotopy between the sphere-valued functions $\bar f|_{A_{r+\delta}'}$
and $\bar f'|_{A_{r+\delta}'}$. The inclusion $\iota: (X,A_{r+\delta}')\hookrightarrow (X,A_r)$ induces a commutative 
diagram\footnote{See~\cite[Chapt. VII, Prop. 4.1]{HuBook} for the naturality of $\delta$ and \cite[Lemma 3.1]{HuBook} for the isomorphism 
$[(X,A), (S^n,*)]\simeq [X/A,S^n]$.}
$$
\begin{array}{ccc}
[A_r, S^{n-1}] & \stackrel{\delta}{\to} & [X/A_r, S^n] \\
\downarrow \iota^*   &  & \downarrow \iota_{/A}^*  \\
\hbox{} [A_{r+\delta}', S^{n-1}] & \stackrel{\delta}{\to} & [X/A_{r+\delta}', S^n] 
\end{array}
$$
and the equality $\iota^*[\bar f]=[\bar f']$ immediately implies that $\iota_{/A}^*$ maps $[f_{/A_r}]$ to $[f_{/A_{r+\delta}'}']$.
So, the inclusion $\iota$ induces an interleaving morphism $u_r: \pi_r\to \pi_{r+\delta}'$ that maps the distinguished element 
to the distinguished element. The other interleaving morphism $v_r: \pi_r'\to\pi_{r+\delta}$ is defined similarly and 
the compositions $v_{r+\delta} u_r$ and $u_{r+\delta} v_r$ behave as required.

We claim that $d_I(\Pi_f, \Pi_{f'})$ is even bounded by $\|f-f'h\|$ where $h$ is any self-homeomorphism of $X$ and hence
the interleaving distance is bounded by the natural pseudo-distance between $f$ and~$f'$.
Let $A_r:=\{x: |f'(x)|\geq r\}$, $A_r':=h^{-1}(A_r)$, $\pi_r$ and $\pi_r'$ be the image of the connecting homomorphism 
$\delta: [A_r, S^{n-1}]\to [X/A_r, S^n]$ and $\delta': [A_r', S^{n-1}]\to [X/A_r', S^n]$ respectively.
The homeomorphism $h$ induces a homotopy equivalence $h: (X, A_r')\to (X, A_r)$ and for any $0<r\leq s$ and
$$
\begin{array}{ccc}
(X,A_s) & \stackrel{i}{\hookrightarrow} & (X, A_r)\\
\uparrow {\scriptstyle{h}} & & \uparrow {\scriptstyle{h}} \\
(X, A_s') & \stackrel{i}{\hookrightarrow} & (X, A_r')\\
\end{array}
$$
induces a commutative diagram on the level of cohomotopy groups. 
Using the naturality of the sequence (\ref{e:les1}), $h^*$ maps $\pi_r$ isomorphically to $\pi_r'$ and it maps 
$[f'_{/A_r}]$ to $[(f'h)_{/A_r'}]$ by definition of the induced map. It follows that $\Pi_{f'}$ and $\Pi_{f'h}$ are isomorphic and
$$
d_{I}(\Pi_f, \Pi_{f'})=d_I(\Pi_f, \Pi_{f'h})\leq \|f-f'h\|.
$$
It is an elementary observation that  $\Pi_f$ and $\Pi_{rf}$ are also isomorphic for any rotation $r$ of $\R^n$.

\heading{Construction of pointed barcode.}
Let $I$ be an interval and $\F$ a field. An \emph{interval module} $C(I)$ is by definition an (unpointed) persistence module 
$(V_r, \varphi_{s,r})_{s\geq r>0}$ such that $V_r\simeq \F$ for $r\in I$, $V_r$ is trivial for $r\notin I$, 
$\varphi_{s,r}$ is the identity if $r,s\in I$ and the zero map otherwise. Any pointwise finite dimensional persistence module
consisting of vector spaces over $\F$ is isomorphic to a direct sum of interval modules, the corresponding intervals as well 
as their multiplicities being uniquely determined~\cite{Boevey:2012,Chazal:2012}.

Let $X,n,f,\F$ be such as in Theorem~\ref{t:computability}.
We will show in Lemma~\ref{l:discret} that there are only finitely many \emph{critical values} 
$s_1,\ldots, s_k$ such that $\varphi_{s,r}$ is an isomorphism whenever $[r,s)$ is disjoint from $\{s_1,\ldots,s_k\}$. 
Cohomotopy groups of finite simplicial complexes are finitely generated and it follows that
$\Pi_f\otimes \F$ is a pointwise finite dimensional pointed persistence module, that is, each $\pi_r\otimes \F$
is a finite dimensional vector space. 

Under these assumptions, we have the following:
\begin{lemma}
\label{l:direct}
The distinguished elements $[f_{/A_r}]\otimes 1$ in $\Pi_f\otimes \F$ generate a direct summand of $\Pi_f\otimes \F$.
\end{lemma}

\begin{proof}
The distinguished submodule generated by the distinguished element is isomorphic to an interval module $C(I)$, 
as it consists of at most one-dimensional vector spaces and $\varphi_{s,r}$ maps a generator to a generator. 
For simplicity, let us denote $[f_{/A_r}]\otimes 1$ by $a_r$ for all $r>0$. 
The corresponding (possibly empty) interval $I$ consists of all $r$ for which $a_r\neq 0$.

If $C(I)=0$, then $C(I)$ is trivially a direct summand. Otherwise $I$ contains a positive $r>0$ and consequently all $t\in (0,r)$.
Let $t>0$ be smaller than any of the critical values $s_i$ of $f$. It follows that $\varphi_{t,s}$ is an isomorphism for all
$0<s\leq t$.  

Choose a decomposition $\Pi_f\otimes \F\simeq \oplus_{\lambda\in\Lambda} C({\lambda})$ of $\Pi_f\otimes \F$ 
into interval modules, where $\Lambda$ is a multiset of intervals. 
The inclusion maps $a_t$ into a finite combination $\sum_{j=1}^k (v_{j})_t$ where $0\neq (v_{j})_t\in C({\lambda_j})_t$ 
for some $\lambda_1,\ldots,\lambda_k\in\Lambda$ (some $\lambda_j$'s may be equal to each other but we assume that the number of intervals
equal to $\lambda_i$ is at most the multiplicity of $\lambda_i$ in $\Lambda$).

We claim that for all $j=1,\ldots, k$, $\lambda_j$ has the form $(0,l_j)$ or $(0, l_j]$ for some $l_j$ such that $\lambda_j\subseteq I$.
Assume, for contradiction, that some $\lambda_j$ does not contain a number $s\in (0,t)$.
Then the projection of $a$ to the direct summand $C({\lambda_j})$ is zero in time $s$ and nonzero in time $t$ which contradicts the
commutativity of 
$$
\begin{array}{ccc}
C(\lambda_j)_s     & \stackrel{\varphi}{\longrightarrow} & C(\lambda_j)_t\\
\uparrow & & \uparrow \\
C(I)_s & \stackrel{\varphi}{\longrightarrow} & C(I)_t\\
\rotatebox{90}{$\in$} && \rotatebox{90}{$\in$}\\
a_s && a_t
\end{array}.
$$
Similarly, $r\in {\lambda_j}$ is impossible for $t<r\notin I$, because $a_r=0$ contradicts $(v_{j})_r\neq 0$. 

Further we claim that at least one $\lambda_j$ is equal to $I$.  Otherwise we could find an $s\in I$ 
that is disjoint from all $\lambda_j$ and derive a contradiction with $0=\sum_j (v_j)_s=\varphi_{s,t} a_t=a_s\neq 0$. 
Suppose, without loss of generality, that $\lambda_k=I$.
Summarizing our construction, we have that $a_t=\sum_{j=1}^k (v_j)_t$ holds for each $t>0$ and $(v_k)_t$ is nonzero iff $t\in I$.

We claim that in the decomposition to interval modules, we may replace $C(\lambda_k)$ with $C(I)$ and obtain another decomposition:
this will prove that $C(I)$ is a direct summand. More formally, we claim that
$$
\Pi_f\otimes \F\simeq C(I)\oplus \left(\bigoplus_{\lambda\in\Lambda\setminus\{\lambda_k\}} C(\lambda)\right)
$$
where $\Lambda\setminus\{\lambda_k\}$ is a multiset where the multiplicity of $\lambda_k$ is one less than its multiplicity in $\Lambda$.
Let $w_r\in (\Pi_f\otimes \F)_r$ be arbitrary. There is a unique decomposition $w_r=\sum_{j=1}^k c_j (v_j)_r + w'$ where $c_j\in \F$ and $w'$
is a combination of elements in the interval modules $C(\lambda)$ for $\lambda\notin \{\lambda_1,\ldots, \lambda_k\}$. 
Another way to write this is $w_r=c_k (v_1+\ldots +v_k)_r + (\sum_{j=1}^{k-1} (c_j-c_k) (v_j)_r)+w'$, or equivalently 
$c_k a_r + (\sum_{j=1}^{k-1} (c_j-c_k) (v_j)_r)+w'$ which yields the projections to the new decomposition.
\end{proof}

We define a \emph{pointed barcode} to be a pair $(\Lambda,I)$ where $\Lambda$ is a multiset of intervals and $I$ 
is an interval that occurs in $B$ at least once.\footnote{If $I$ occurs $k>1$ times in $B$, 
then we cannot distinguished which of the $k$ copies of $I$ is the distinguished intervals.}
We may represent $\Pi_f\otimes \F$ via a pointed barcode, the multiset $\Lambda$ corresponding to the unpointed decomposition 
of $\Pi_f\otimes \F$ into interval modules, and the interval $I$ corresponding to the direct summand generated by the distinguished elements. 

The usual bottleneck distance generalizes to this structure as follows. 
\begin{definition}
%
The bottleneck distance 
$d_\mathcal{B}((\Lambda_1, I_1), (\Lambda_2, I_2)$ between two pointed barcodes is the infimum of all $\delta$ such that there exists 
a matching\footnote{For a rigorous definition of barcode matching, see e.g.~\cite{Bauer:2013}. 
Note that there a barcode is defined to be a {set} that \emph{represents} the multiset $\Lambda_i$.}
between $\Lambda_1$ and $\Lambda_2$ such that
\begin{itemize}
\item All intervals of length at least $2\delta$ are matched,
\item The matching shift end-points of intervals at most $\delta$-far,
\item If either of the distinguished intervals $I_1, I_2$ is matched, then both of them are matched and they are matched together.
\end{itemize}
\end{definition}
Note that if both distinguished bars have lengths smaller than $2\delta$, then they are allowed to be unmatched.
The next lemma addresses the stability of the bottleneck distance of pointed modules.
\begin{lemma}
Let $f,f'$ be as in Theorem~\ref{t:computability} and let $(\Lambda,I)$ and $(\Lambda',I')$ be the pointed barcode representing 
$\Pi_f\otimes \F$ and $\Pi_{f'}\otimes \F$, respectively. 
Then the bottleneck distance $d_{\mathcal{B}}((\Lambda,I), (\Lambda', I'))$ is bounded by
the interleaving distance $d_I(\Pi_f\otimes \F, \Pi_{f'}\otimes \F)$.
\end{lemma}
The chain of inequalities 
$$d_{\mathcal{B}}((\Lambda, I), (\Lambda', I'))\leq d_I(\Pi_f\otimes \F, \Pi_{f'}\otimes \F)\leq d_I(\Pi_f, \Pi_{f'})\leq \|f-f'\|$$ 
then implies the stability of the bottleneck distance with respect to perturbations of $f$.
\begin{proof}
Assume that $\Pi_f\otimes \F$ and $\Pi_{f'}\otimes \F$ are $\delta$-interleaved
and let $u_r: (\pi_r\otimes \F, [f_{/A_r}]\otimes 1)\to (\pi_{r+\delta}', [f'_{/A_{r+\delta}'}]\otimes 1)$ 
and $v_r: (\pi_r'\otimes \F, [f'_{/A_{r}'}]\otimes 1)\to (\pi_{r+\delta}, [f_{/A_{r+\delta}}]\otimes 1)$ 
be the families of interleaving morphisms. 

Using the decompositions from Lemma~\ref{l:direct}, we have that 
$\Pi_f\otimes \F\simeq C(I)\oplus D$ and $\Pi_{f'}\otimes \F\simeq C(I')\oplus D'$ where 
$C(I), C(I')$ are the distinguished submodules and $D$, $D'$ their complements. The interleaving morphisms $u_r$ and $v_r$
map $C(I)_r$ to $C(I')_{r+\delta}$ and $C(I')_r$ to $C(I)_{r +\delta}$, respectively. 
The interleaving $u_r$, $v_r$ then induce the maps
$u_r^{(1)}: C(I)_r\to C(I')_{r+\delta}$ between the distinguished submodules and 
$u_r^{(2)}: (\pi_r\otimes \F)/C(I)_r \to (\pi_{r+\delta}'\otimes \F)/C(I')_{r+\delta}$ between the factor modules,
and similarly, $v_r$ induces analogous maps $v_r^{(1)}$ and $v_r^{(2)}$. 
The factor modules $(\Pi_f\otimes\F)/C(I)$ and $(\Pi_{f'}\otimes\F)/C(I')$ are isomorphic to the complementary modules $D$ and $D'$ 
respectively.

The families $u_r^{(1)}$ and $v_r^{(1)}$ define a $\delta$-interleaving between the distinguished submodules that can be represented
each by at most one bar in the barcode representation. By the standard stability theorem for unpointed barcodes~\cite[Thm. 6.4]{Bauer:2013}, 
there exists a $\delta$-matching between $\{I\}$ and $\{I'\}$.\footnote{Here we have the convention that if $I$ is empty, then $\{I\}$ represents the empty multiset.}
Similarly, $u_r^{(2)}$ and $v_r^{(2)}$ define a $\delta$-interleaving between the quotients $(\Pi_f\otimes \F)/C(I)$
and $(\Pi_{f'}\otimes \F)/C(I')$ that are represented by the multiset of intervals complementary to $I$ and $I'$ respectively, and they induce
a $\delta$-matching between these complementary barcodes. The disjoint union of these $\delta$-matchings gives an upper bound
on the bottleneck distance between the pointed barcodes.
\end{proof}

\section{Proof of Theorem B}
\heading{Star, link and subdivision of simplicial complexes.}
Let $A\subseteq X$ be simplicial complexes. We define the $star(A,X)$ to be the set of all faces of all simplices in $X$ that have
nontrivial intersection with $|A|$, and $link(A,X):=\{\sigma\in star(A,X)\,|\,\sigma\cap |A|=\emptyset\}$. Both $star(A,X)$ and $link(A,X)$ are simplicial complexes.
The difference $star^\circ(A,X):=|star(A,X)|\setminus |link(A,X)|$ is called the \emph{open star}.
A~simplicial complex $X'$ is called a \emph{subdivision} of $X$ whenever $|X'|=|X|$
and each $\Delta'\in X'$ is contained in some $\Delta\in X$.
If $a\in |X|$, than we may construct a subdivision of $X$ by replacing the unique $\Delta$ containing $a$ in its interior
by the set of simplices $\{a,v_1,\ldots, v_k\}$
for all $\{v_1,\ldots, v_k\}$ that span a face of $\Delta$, and correspondingly subdividing each simplex containing $\Delta$.  This process is called \emph{starring} $\Delta$ at $a$.
If we fix a point $a_\Delta$ in the interior of each $\Delta\in X$, we may construct a \emph{derived subdivision} $X'$ by starring each
$\Delta$ at $a_\Delta$, in an order of decreasing dimensions.

\heading{Computability of cohomotopy groups.}
The crucial external ingredient for the proof is the polynomial-time
algorithm for computing cohomotopy groups  \cite[Theorem 1.1]{post},
see \cite[Theorem 3.1.2]{Krcal-thesis} for the running time analysis.
\begin{proposition}[{\cite[Theorem 3.1.2]{Krcal-thesis}}]   
\label{p:compute2}
For every fixed $k\ge 2$, there is a polynomial-time algorithm
that, \begin{enumerate}
\item given a~finite simplicial complex (or simplicial
set) $X$ of dimension $k$ and a $d$-connected $S$, where $k\le
2d$, computes the isomorphism type of $[X,S]$ as a finitely generated
Abelian group.%

\item When, in addition, a simplicial map $f\:X\to S$ is given, the
algorithm expresses $[f]$ as a linear combination of the generators.

\item Finally, when, in addition, a simplicial map $g\:X\to X'$ with $\dim X'\le 2d$ is given, 
the algorithm computes the induced homomorphism $$g^*\:[X',S]\to[X,S].$$
\end{enumerate}
\end{proposition}
The item 3 above is not explicitly stated in \cite[Theorem 3.1.2]{Krcal-thesis} but the computation simply amounts to composing the simplicial map $g$ with the representatives of the generators computed by item 1  (see \cite[Theorem 3.6.1]{Krcal-thesis} for the details on the representation) and applying item 2 on the composition. See also \cite[Proof of Theorem 1.4]{polypost} for an explicit computation of an induced homomorphism.

We will split the proof of Theorem B in two parts: first we show the polynomial computability
of the cohomotopy persistence module and then the polynomial complexity of the computation of the pointed barcode associated to $\Pi_f\otimes \F$.
In the  analysis of the running time, $n$ is supposed to be fixed and the polynomiality is with respect to the size of the input that defines the 
simplicial complex $X$ and the simplexwise linear function $f$. The simplicial complex is encoded by listing all its simplices so the size of the input size is at least the number of simplices $|X|$. We do not present an estimate of how the complexity depends on $n$.

\begin{proof}[Proof of Thm. B, part (1): Computability of the cohomotopy persistence module.]
First we focus on the computation of each particular
$\pi:=\pi_r$ and $[f_{/A}]:=[f_{/A_r}]$ for any fixed $r>0$.   
We need the following  segment of the long exact sequence of cohomotopy groups
\cite[Chapter VII]{HuBook}
\begin{equation}\label{e:ex}
[A,S^{n-1}]\stackrel{\delta}{\to}
[X/A,S^{n}]\stackrel{j^*}{\to}
[X,S^n].\end{equation}
The desired $\pi=\im\delta$ can be computed in various ways:
we will use the exactness at $[X/A,S^n]$, that is, $\im\delta = \ker j^*$.

The outline  of the algorithm is as follows:
\begin{enumerate}
\item Discretize the pair $(X,A)$ by a homotopy equivalent
pair of simplicial complexes $(X',A')$. 
\item Using simplicial approximation theorem, discretize the map $\bar f\:A\to S^{n-1}$ by a simplicial
map $\bar f'\:A'\to \Sigma^{n-1}$ where $\Sigma^{n-1}$ is the boundary of the $n$-dimensional
cross polytope.

\item Construct a discretization $f'\:X'\to \Cone \Sigma^{n-1}$
of $f$ as an extension  
$$\bar f'\:A'\to \Sigma^{n-1}\subseteq \Cone \Sigma^{n-1}$$
by sending each vertex in $X'\setminus A'$
to the apex of the cone. Use the simplicial quotient operation on  $f'\:(X', A')\to (\Cone\Sigma^{n-1},\Sigma^{n-1})$
to get the discretization 
$$f'_{/A}\:X'/A'\to\Cone \Sigma^{n-1}/\Sigma^{n-1}=: \Sigma^n$$ 
of $f_{/A}\:X/A\to S^n$.
\item Apply Proposition~\ref{p:compute2} (1) to get $[X'/A',\Sigma^n]$
and $[X',\Sigma^n]$, Proposition~\ref{p:compute2} (2) to get $[f'_{/A}]\in[X'/A',\Sigma^n]$
and Proposition~\ref{p:compute2} (3) to obtain the induced homomorphism
$j'^*\:[X'/A',\Sigma^n]\to[X',\Sigma^n]$.
\item Compute the kernel of 
$ j^*$ and express the element $[f'_{/A}]$ in terms of the generators
of $\ker j^*$ (\cite[Lemma 3.5.2]{Krcal-thesis}).
\end{enumerate}    

The details follow.
\heading{Step 1.}
First we need to ``discretize'' the pair $(X,A)$ by a homotopy equivalent
pair of simplicial complexes $(X',A')$. 

As in~\cite[Proof of Theorem 1.2]{nondec} we compute a
subdivision $X'$ of $X$ such that for each simplex $\Delta\in
X'$ we have that $\min_{x\in\Delta}|f(x)|$ is attained in a vertex
of $\Delta$. This can be done by starring each $\Delta$ in $\argmin_
{x\in\Delta}|f(x)|$ whenever it belongs to the interior of $\Delta$. The polynomial-time computability (when $m$ is fixed)  of $\argmin_
{x\in\Delta}|f(x)|$ is our
only requirement on the norm $|\cdot|$ in $\R^n$; it is satisfied
for  all norms $\ell_1,\ell_\infty$ (via a~linear program with a~fixed number of variables and inequalities) 
and $\ell_2$ (Lagrange multipliers). We will refer
to the values $\min_{x\in\Delta}|f(x)|$ for $\Delta\in X$
as \emph{critical values} of $f$. Moreover, for the next step we will require that for each component $f_i$ of $f$ 
the preimage $f_i^{-1}(0)$ intersects each edge of $X'$ in a vertex (or not at all). 
Thus for each $i=1,\ldots,n$ we  do starring in an arbitrary order of  each edge $ab$ in $f_i^{-1}(0)\cap ab$ 
whenever the intersection consist of a single interior point of the edge. 
Note that this does not destroy the property that the minimum of $|f|$ on each simplex is attained in a vertex.
In the end, the number of starring is bounded by a constant multiple of the number of simplices of $X$. 

We define the discretization $A'\subseteq X'$ of $A\subseteq X$ as in the following lemma.
\begin{lemma}\label{l:discret}
Let $A'$ be the simplicial subcomplex of $X'$ that is spanned
by the vertices $x$ of $X'$ such that $|f(x)|\ge r$. Then $A'$
is a strong deformation retract of $A=A_r:=\{x\:|f(x)|\ge  r\}$. 
\end{lemma}
\begin{proof}
The strong deformation retraction (that is, a~map $H\:A\times[0,1]\to
A$ with $H(\cdot,0)=\id$, $\im H(\cdot,1)\subseteq A'$ and $H(\cdot,t)$\
being identity on $A'$) is constructed simplexwise. Namely for each
$\Delta\in X'$ it is the straightline homotopy between identity
on $\Delta\cap A$ and $p|_{\Delta\cap A}$ where $p$ is the projection
of $\Delta\cap star^\circ(\sigma)$ onto the maximal face $\sigma$ of $\Delta$ that is contained in $A'$. We claim that the image of $H$ is contained in $A$. The image of $H(a,\cdot)$ is certainly contained in the segment between $a$ and $p(a)$ which is a subsegment of a segment between $s$ and $p(a)$ where $s$ is a unique point on the face of $\Delta$ that is complementary to $\sigma$. Because $s$ is not contained in $A$ and because of the convexity of the complement of $A$ (that follows from the convexity of the norm), each point between $a$ and $p(a)$ has to be contained in $A$. 
Finally, it is routine to check that the definition of $H$ on
each $\Delta$ is compatible with its definition on every face
$\Delta' < \Delta$. 
\end{proof}

\heading{Step 2.} Next we ``discretize'' the map $\bar f\:A\to S^{n-1}$ by a simplicial
map $\bar f'\:A'\to \Sigma^{n-1}$ where $\Sigma^{n-1}$ is the boundary of the $n$-dimensional
cross polytope---a convenient discretization of $S^{n-1}$. By discretization we mean that we get commutativity up to homotopy in the diagram $$\xymatrix{
A\ar[r]^{\bar f}\ar[d]^{p}&S^{n-1}\ar[d]\\
A'\ar[r]^{\bar f'}&\Sigma^{n-1}}
$$
where the vertical map $p\:A\to A'$ is the homotopy equivalence from the previous step and $S^{n-1}\to \Sigma^{n-1}$ is the homeomorphism defined by $x\mapsto x/|x|_1$.

The construction of the simplicial approximation $f'\:A'\to \Sigma^{n-1}$ of $\bar f$ follows exactly the procedure from \cite[Proof of Theorem 1.2]{nondec}.
Due to the second subdividing step from above, for each vertex $v\in X'$, there is $i\in\{1,\ldots,n\}$ such that $f_i$ has a constant sign $s\in\{+,-\}$ on $star^\circ(v,A')$. We prescribe $f'(v):=s e_i$ where $s e_i$ is a vertex of $\Sigma^{n-1}$. By the simplicial approximation theorem ($f$ maps each  $star^\circ(v,A')$ into \(star^\circ(f'(v),\Sigma^{n-1})\)), the map $f'$ is homotopic to the map $A'\to \Sigma^{n-1}$ defined by $x\mapsto f(x)/|f(x)|_1$. By the deformation retraction above, $f'p$ is also homotopic to $A\to \Sigma^{n-1}$ defined again by $x\mapsto f(x)/|f(x)|_1$.


\heading{Step 3.} Next we construct a~simplicial approximation 
$f'\:X'\to \Cone \Sigma^{n-1}$ 
of $f$ as an extension of 
$$\bar f'\:A\to \Sigma^{n-1}\subseteq \Cone \Sigma^{n-1}$$ 
by sending each vertex in $X'\setminus A'$ to the apex of the cone. 
Further, we use the simplicial ``quotient operation'' on  
$$f'\:(X', A')\to (\Cone\Sigma^{n-1},\Sigma^{n-1})$$
to get the simplicial approximation 
$$f'_{/A'}\:X'/A'\to\Cone \Sigma^{n-1}/\Sigma^{n-1} \quad \text{ of } \quad f_{/A}\:X/A\to S^n.$$

The quotient operation, strictly speaking, exists for simplicial sets but not for simplicial complexes. However,
all simplicial complexes and maps can be canonically converted
into simplicial sets and maps of simplicial sets (\cite[Section
2.3]{polypost}) after fixing  arbitrary orderings of all the
vertices of each simplicial complex that are compatible with the given maps.  
First we choose an ordering of the vertices of  $\Cone\Sigma^{n-1}$ arbitrarily,
and then an ordering of the vertices of $X'$ such that $f'(u)<f'(v)$
implies $u<v$.

By construction, $[f'_{/A'}]=\delta\big([\bar f']\big)$ and thus $[f'_{/A'}]$ is a~simplicial approximation of
$\delta [\bar{f}]$.


\heading{Step 4.} Apply Proposition~\ref{p:compute2} (1) to get 
$[X'/A',\Sigma^n]$ and $[X',\Sigma^n]$ where 
$$\Sigma^n=\Cone\Sigma^{n-1}/\Sigma^{n-1}.$$ 
Then apply Proposition~\ref{p:compute2} (2) to get $[f'_{/A}]\in[X'/A',\Sigma^n]$.
The simplicial quotient map $j'\:X'\to X'/A'$ is a discretization of $j\:X\to X/A$ and
we use Proposition~\ref{p:compute2} (3) to obtain the induced homomorphism 
$$j'^*\:[X'/A',\Sigma^n]\to[X',\Sigma^n].$$
The polynomial running time of this step amounts to Proposition~\ref{p:compute2}.

\heading{Step 5.} Finally, we compute $\pi$ as the kernel of 
$ j^*$ and express the element $[f'_{/A}]$ in terms of the generators of $\ker j^*$.
The correctness and polynomial running time of this step
amounts to \cite[Lemma 3.5.2]{Krcal-thesis}.

Further, assume that $r$ is not fixed.

\begin{lemma}
\label{l:critical}
If an interval $[r,s)\subseteq \R^+$ is disjoint from $\{\min_{x\in\Delta} |f(x)|: \Delta\in X\}$, then
$\varphi_{s,r}: (\pi_r, [f_{/A_r}])\to (\pi_s, [f_{/A_s}])$ is an isomorphism.
\end{lemma}
\begin{proof}
Let $[r,s)$ be disjoint from $\{\min_{x\in\Delta} |f(x)|: \Delta\in X\}$. Then for each vertex $v\in X'$ we have that 
$|f(v)|\geq r$ iff $|f(v)|\geq s$ and so both $A_r$ and $A_s$ deformation retract to $A'$ by Lemma~\ref{l:discret}. 
Thus the inclusions 
$$i_r\:(X,A')\hookrightarrow (X,A_r)\quad\text{ and }\quad i_s\:(X,A')\hookrightarrow  (X,A_s)$$ 
induce isomorphisms of the
pointed cohomotopy groups 
$$(\pi_r, [f_{/A_r}])\to (\pi', [f_{/A'}]) \quad \text{ and } \quad  (\pi_s, [f_{/A_s}])\to (\pi', [f_{/A'}])$$ 
where $\pi'$ is the corresponding subgroup of $[X/A', S^n]$. 
The inclusion 
$$i: (X,A_s)\hookrightarrow (X,A_r)$$
satisfies $i_r\circ i=i_s$ which immediately implies the isomorphism.
\end{proof}
As follows from Lemma~\ref{l:critical}, the homotopy type of $A_r$
can only
change when $r$ passes through one of the critical values $s_1,\ldots,s_k$
of $f$. Therefore, we only have to compute groups $\gr_{r_0},\ldots,\gr_{r_k}$
for arbitrary values $r_0,\ldots,r_k$ such that  $r_0<s_1<r_1<\ldots<r_{k-1}<s_k<r_k$.
The number of critical values is bounded by the number of simplices
of $X$ therefore this can be done in polynomially many repetitions of the above algorithm.

The remaining step is to compute the homomorphisms $\varphi_{i}=
\varphi_{r_{i+1},r_{i}}$ induced by the quotient maps $X/A_{r_{i+1}}\to
X/A_{r_i}$  for \(i=0,\ldots,k\). This is another application of Proposition~\ref{p:compute2} (3) on the discretization
$X'/A'_{r_{i+1}}\to X'/A'_{r_i}$ of the above quotient map.

\end{proof}

\begin{proof}[Proof of Theorem B, part (2): Computation of the pointed barcode.]
Assume that the isomorphism type of $\Pi_f$ has been computed and is represented as a sequence of pointed Abelian groups
$\pi_{r_i}\stackrel{\varphi}{\to} \pi_{r_{i+1}}$ and an initial element $[f_{/A}]\in \pi_{r_0}$. Tensoring the cohomotopy persistence
module with $\Q$ converts each $\Z$-summand of the Abelian groups into a $\Q$-summand and kills the torsion, while tensoring with 
a~finite field $F$ of characteristic $p$ converts each $\Z$-summand and each $\Z_{p^k}$-summand into an $F$-summand
and kills all $Z_{(p')^{k'}}$ for $p'\neq p$. The induced $\F$-linear maps 
$$
\varphi_{r_{i+1}, r_i}^\F: \pi_{r_i}\otimes \F\to \pi_{r_{i+1}}\otimes \F
$$ 
can easily be represented via matrices,
if the action of $\varphi_{r_{i+1}, r_i}$ on generators has been precomputed. The number of critical values $s_1,\ldots, s_k$
is bounded by the size of the input data defining the simplicial complex $X$. 
Each interval in the pointed barcode representation is either $(s_i, s_j]$ or $(0, s_j]$
for some $1\leq i,j\leq k$ and the number of such pairs is bounded by $k^2$. 
Finally, the multiplicity of the interval spanned between $s_i$ and $s_j$ can be computed from a simple rank formula 
$$
\mathrm{rank}\, \varphi_{r_{i-1}, r_j}^\F   - \mathrm{rank}\, \varphi_{r_{i}, r_j}^\F +
\mathrm{rank}\, \varphi_{r_{i}, r_{j-1}}^\F - \mathrm{rank}\, \varphi_{r_{i}, r_{j-1}}^\F
$$
and similarly for the pairs $0,s_j$. Each rank computation has polynomial complexity with respect to the dimension of the matrices.
These dimensions are bounded by the ranks of $\pi_r$, which in turn depend polynomially on the number of simplices in $X$~\cite[Theorem 3.1.2 and  Chapter 1.1.2]{Krcal-thesis}. 
The distinguished barcode is empty iff $[f_{/A_{r_0}}]\otimes 1$ is trivial, and otherwise spanned between $0$ 
and the minimal $s_i$ such that $[f_{/A_{r_i}}]\otimes 1$ is trivial.
\end{proof}

\section{Proof of Theorem C}
\label{a:cobordism2perturbation}
We start with some definitions and simple statements from the field of differential topology. The domain $X$ in Theorem C is 
assumed to be a~smooth manifold, possible with non-empty boundary: in that case, $X\times [0,1]$ is a manifold with corners.
\label{s:preliminaries}
\heading{Manifolds with corners.} A smooth $n$-manifold with corners is a second-countable Hausdorff space $M$ with
an atlas consisting of charts $\varphi_a: U_a\to [0,\infty)^k\times\R^{n-k}$, where $U_a\subseteq M$ are open,
$\{U_a\}_a$ is a covering of $M$ and the transition maps $\varphi_a\circ\varphi_b^{-1}$ are smooth.
Common notion of smooth maps, tangent spaces and diffeomorphism easily generalize to manifolds with corners, see~\cite{Joyce:2012} 
for a detailed exposition. For each $x\in M$, the \emph{depth} of $x$ is equal to $l$ s.t. for some chart $\varphi_a$ the image
$\varphi_a(x)\in [0,\infty)^k\times \R^{m-k}$ has exactly $l$ coordinates among the first $k$ coordinates equal to zero:
this is independent on the choice of the chart. If the depth is at most $1$ for
all $x$, then this reduces to the common notion of a smooth manifold with boundary. 
We will use the notation 
$$\partial_l M:=\{x\in M\,|\,\mathrm{depth}(x)=l\}.$$
Its closure $\overline{\partial_l M} = \cup_{j\geq l}\,\partial_j M$ is naturally an $n-l$ dimensional manifold with corners. 

The category of manifolds with corners is closed with respect to products.
In this work, we will only consider (sub)manifolds with corners of depth at most 2:
they naturally arise as ``regular'' preimages of submanifolds with boundary in
a manifold with boundary. One example is the case of $f^{-1}[0,\infty)$ for smooth $f: M\to \R$ such that both $f$ and 
$f|_{\partial M}$ are transverse to $0$.

A manifold will refer to a smooth manifold with (possibly non-empty) boundary.

%
\heading{Submanifolds.} If $M$ is a smooth $m$-manifold (or $m$-manifold with corners), then a submanifold $N$ will refer
to a smooth \emph{embedded submanifold} (with corners).
%
If $M$ is a smooth manifold, then a \emph{neat submanifold} is an embedded $k$-dimensional submanifold $N$ 
such that $\partial N=N\cap\partial M$ and for each $x\in \partial N$ there exists an $M$-chart $\psi: U_x\to[0,\infty)\times\R^{m-1}$ 
such that $\psi^{-1}([0,\infty)\times\R^{k-1}\times\{0\}^{m-k})=N\cap U_x$.
We will extend this definition to manifolds with corners.
\begin{definition}
Let $M$ be an $m$-manifold with corners. A neat $k$-submanifold $N$ with corners is a smooth embedded submanifold such that
$\partial_j N=N\cap\partial_j M$ for each $j$, and for each $x\in\partial_j N$ there exists an $M$-chart 
$\psi: U_x\to [0,\infty)^j\times\R^{m-j}$ such that $\psi^{-1}([0,\infty)^j\times\R^{k-j}\times\{0\})=N\cap U_x$.
\end{definition}

If $N$ is a submanifold of $M$, then its boundary $\partial N$ does not need to be equal to the topological boundary of $N\subseteq M$.
We will use the notation $\manbound N$ for the manifold-boundary and $\partial N$ for the topological boundary of $N$ in $M$
wherever some ambiguity will be possible.
\heading{Transversality.}
The transversality theorem says that, roughly speaking, for smooth maps $M\to N$ and submanifolds $A\subseteq N$, transversality to $A$
is a generic property. If $M$ is compact and $A\subseteq N$ is closed in $N$, 
then the subspace of all smooth maps $M\to N$ transverse to $A$ is both dense and open 
(see \cite[Theorem 2.1]{Hirsch:76} for the case of boundary-free manifolds). 
If $f:M\to N$ is a smooth map between manifolds, $n\in N\setminus\partial N$  and both $f$ and $f|_{\partial M}$ are transverse to $\{n\}$
(equivalently, $n$ is a regular value of both maps), then $f^{-1}(n)$ is a neat submanifold of $M$~\cite[p. 27]{Kosinski:2007}.
Similarly, if $N'$ is a neat submanifold of $N$ with corners and $f|_{\overline{\partial_j M}}$ is transverse to $N'$ for each $j$, 
then $f^{-1}(N')$ is a neat submanifold of $M$ with corners. 

\heading{Framed submanifolds.} Assume that $X$ is a smooth oriented $m$-manifold. 
Let $S\subseteq X$ be a smooth $(m-n)$-submanifold of $X$. A framing of $S$ is
a trivialization $T$ of the $n$-dimensional quotient bundle $(TX)|_S/TS$:
that is, $T(x):=(T_1(x),\ldots, T_{n}(x))$ is a basis of $T_x X/T_x S$ in each $x\in S$.
Any choice of a Riemannian metric on $X$ induces an isomorphism $T_x X/T_x S\simeq N_x X$, 
where $N_x S$ is the space of all vectors in $T_x X$ orthogonal to $T_x S$,
so a framing can be understood as a trivialization of the normal bundle.

Assume that $X$ is a smooth manifold with boundary, $f: X\to\R^n$ is smooth and that $0$ is a regular value of $f$.
Then $f^{-1}(0)$ is naturally a framed $(m-n)$-submanifold, $T_i(x)$ being the unique element of $T_x X/T_x f^{-1}(0)$
mapped by $df$ to the $i$th basis vector $e_i\in T_0\R^n$. We will denote these vectors by $f^*(e_i)$.
Such framing uniquely determines---and is uniquely determined by---the differential $df|_{f^{-1}(0)}$.
If $0$ is also a regular value of $f|_{\manbound X}$, then $f^{-1}(0)$ is a~{neat submanifold} of $X$ 
and $\partial f^{-1}(0)= f^{-1}(0)\cap \manbound X$ is an $m-n-1$ dimensional submanifold of $\manbound X$ 
with an $n$-framing induced by $f|_{\manbound X}$. 

Assume $S\subseteq X$ a neat framed submanifold.
For $x\in S\cap \manbound X$, we can naturally identify $T_x X/T_x S\simeq T_x (\manbound X)/T_x (\manbound S)$~\cite[p. 53]{Kosinski:2007},
so the framings on $\manbound S$ induced by $f$ and $f|_{\manbound X}$ are compatible.
Any Riemannian metric on $X$ in which $S$ intersects $\manbound X$ orthogonally ($N_x S\subseteq T(\manbound X)$ for each $x\in\partial S$) 
can be used to represent the framings of $S$ and $S\cap\manbound X$ as compatible normal vectors to $TS$, resp. 
$T(\manbound S)$. In particular, given such metric, if $0$ is a regular value of both $f$ and $f|_{\manbound X}$, 
then the geometric representation of the framing of $f^{-1}(0)$ induced by $f$ restricts on the boundary to the geometric representation 
of the framing induced by $f|_{\manbound X}$. 
\heading{Replacing $r$-perturbations by homotopy perturbations.}
In this paragraph we will derive a~smooth analogue of Lemma~\ref{l:perturb-ext}.
\begin{definition}
Let $X$ be a smooth manifold, $A\subseteq X$ closed and $f: (X,A)\to (\R^n, \R^n\setminus\{0\})$. 
A function $h: (X,A)\to (\R^n, \R^n\setminus\{0\})$
will be called a \indef{regular homotopy perturbation} of $f$, if $h$ is smooth, $h$ and $h|_{\manbound X}$ are transverse to $0$
and $f$ is homotopic to $h$ as maps $(X,A)\to (\R^n, \R^n\setminus\{0\})$.
\end{definition}
\begin{lemma}
\label{l:homotopy2pert_smooth}
Let $X$ be a smooth manifold, $A$ its closed subset and $f: X\to\R^n$ such that $A=\{x: \,\, |f(x)|\geq r\}$. Then 
$$
\ZrFr(f)=\{ (h^{-1}(0), h^*(e_i)) \,\,|\,\, h\,\,\text{is a regular homotopy perturbation of } \,f \}.
$$
\end{lemma}
\begin{proof}
The inclusion $\subseteq$ follows from the fact that a regular $r$-perturbation $g$ of $f$ is straight-line homotopic to
$f: (X,A)\to (\R^n, \R^n\setminus\{0\})$.

The other inclusion will also be proved analogously to Lemma~\ref{l:perturb-ext}.
Choose an $\epsilon>0$ so that $\min_{x\in A} |h(x)|\geq 2\epsilon$ and $\epsilon$ is a regular value of 
$|h|$: then $Y:=|h|^{-1}[\epsilon,\infty)$ is a smooth manifold with corners in $\partial Y\cap\manbound X$. 
We have assume that $(f/|f|)|_A$ is homotopic to $(h/|h|)|_A$, so $f|_A$ and $h|_A$ are homotopic as maps to $\R^n\setminus\{0\}$.
The partial homotopy of $h$ on
$|h|^{-1}(\epsilon)\cup A$ that is stationary on $|h|^{-1}(\epsilon)$ and equal to the given homotopy $h|_A\sim f|_A$
on $A$ can be extended to a homotopy $H: Y\times[0,1]\to\R^n\setminus\{0\}$ by the homotopy extension property, 
so that $H(\cdot, 0)=h|_Y$ and $H(\cdot, 1)$ coincides with $f$ on $A$. Without loss of generality, we may assume that
$H$ is smooth (compare~\cite[Col. III 2.6]{Kosinski:2007}).

We define a map $e: X\to\R^n$
that equals $H(\cdot, 1)$ on $Y$ and $h$ on $|h|^{-1}[0,\epsilon]$. 
This map $e$ is an extension of $f|_A$ and equals $h$ in some neighborhood of $h^{-1}(0)$. It is smooth everywhere except possibly on
$|h|^{-1}(\epsilon)$: if $e$ is not smooth, we may slightly perturb it in a neighborhood of $|h|^{-1}(\epsilon)$ without changing 
its values in a neighborhood of $0$ or on $A$: assume further that $e$ is smooth.
As we have seen in the proof of Lemma~\ref{l:perturb-ext}, some positive scalar multiple
$\chi(x)\,e(x)=:\tilde{g}(x)$ satisfies $\|\tilde{g}-f\|<r$. The map $\chi$ can be chosen to be smooth. 
Multiplying $\tilde{g}(x)$ by a positive smooth
$(0,1]$-valued function that equals $1/\chi$ in $e^{-1}[0,\delta]$ and $1$ in $e^{-1}[2\delta, \infty)$, 
we obtain a map $g$ that still satisfies $\|g-f\|<r$, if $\delta$ is small enough.
This map $g$ coincides with $h$ in a neighborhood of $h^{-1}(0)$, hence both $g$ and $g|_{\manbound X}$ are transverse to $0$ 
and $g$ induces the same framing of the zero set as $h$. 
\end{proof}
For the rest of the proof of Theorem C, we will use the characterization of $\ZrFr(f)$ from Lemma~\ref{l:homotopy2pert_smooth}.
Namely, $\ZrFr(f)$ will refer to framed zero sets of regular homotopy perturbation of $f$.

\heading{Product neighborhoods.} If $X$ is a manifold with boundary and $C$ is a neat framed submanifold of $X$
of codimension $n$, then there exists a neighborhood $\mathcal{N}$ of $C$ diffeomorphic to $C\times\R^n$ 
(see~\cite[Theorem 4.2]{Kosinski:2007} for a slightly more general statement). 
Similarly, it holds that a framed neat submanifold with corners
of a manifold with corners has a neighborhood diffeomorphic to a product of the submanifold and $\R^n$. 
\begin{lemma}
\label{l:product_n}
Let $X$ be a manifold, possibly with corners, and $C$ be a neat framed submanifold of $X$ of codimension $n$. Then there exists
a neighborhood $\mathcal{N}$ of $C$ diffeomorphic to $C\times B^n$ where $B^n$ is the closed $n$-ball, and a function $F: \mathcal{N}\to\R^n$
so that $F$ and $F|_{\manbound X}$ are transverse to $0$ and $F^{-1}(0)$ equals $C$ with the induced framing.
\end{lemma}
The set $\mathcal{N}$ will be referred to as the \emph{product neighborhood} of $C$. 
We only sketch the proof, as it is an easy generalization of standard concepts.
\begin{proof}
In case of a manifold $X$ with no boundary, the diffeomorphism $S\times B^n\simeq \mathcal{N}$ is constructed via geodesic flow
of the geometric framing vectors (assuming the choice of a smooth metric). The function $F$ is then defined,
after identifying $\mathcal{N}$ with $C\times B^n$, as the projection $C\times B^n\to B^n$.

If $X$ has a boundary, then we need to choose the metric carefully so that the geodesic flow of the geometric framing vectors in
$\manbound X\cap C$ stays in $\manbound X$. This can be done as follows. First we choose an arbitrary smooth metric on $\manbound X$
and a vector field $v$ on $\manbound X$ so that $v$ points inwards to $X$ and $v|_{\manbound C}$ points inwards to $C$. 
We extend $v$ to a~smooth vector field nonzero on some neighborhood of $\manbound X$: the flow of $v$ defines 
a~collar neighborhood of $\manbound X$ diffeomorphic to $\manbound X\times [0,1]$. 
We extend the metric on $\manbound X$ to a product metric on this collar neighborhood: finally, we extend this arbitrarily 
to a smooth metric on all of $X$.  In this metric, $C$ intersects 
$\manbound X$ orthogonally and geodesics generated by tangent vectors in $\manbound X$ remain in $\manbound X$.
We use this metric to convert the $C$-framing to a~geometric framing. Its geodesic flow is well defined and
generates a~diffeomorphism 
from $C\times B^n(\epsilon)$ to a product neighborhood of $C$ for $\epsilon$ small enough.

For a general manifold with corners, we proceed analogously but need to define the metric via a longer chain of extensions.
First we define the metric on $\partial_l X$ for the maximal $l$ for which $\partial_l X$ is nontrivial.  
For each component of $\partial_l X$,
we extend the metric step by step to all components of $\partial_{l-1} X$ that meet at the given component of $\partial_l X$
as follows. First, the intersection $C\cap \overline{\partial_{l-1} X}$ should intersect $\partial_l X$ orthogonally. 
Next, we require that 
for each component $U_{ij}$ of $\partial_{l-1} X$ meeting a given component $V_j$ of $\partial_l X$, there exist some neighborhood of
$V_j$ in $\overline{U_{ij}}$ on which the metric is a product metric. Inductively, we construct the metric on all $\partial_j X$ for $j\leq l$.
Then the geometric framing vectors of $C$ in $x\in C\cap\partial_j X$ are in $T(\partial_j X)$ and there exists an $\epsilon>0$ so that
the geodesic flow of the framing vectors of $C\cap\partial_j X$ stays in $\partial_j X$ for $t\in [0,\epsilon]$. 
We define the product neighborhood via geodesic flow as before.
\end{proof}

\heading{Pontryagin-Thom construction.}
Let $X$ be a smooth $m$-manifold with boundary and $N_1$, $N_2$ be two neat framed $(m-n)$-submanifolds of $X$. We say that $C$ is a \emph{framed cobordism} between 
$N_1$ and $N_2$ if $C$ is a framed neat $(m-n+1)$-submanifold of $X\times [0,1]$ (with corners in $C\cap (\manbound X\times \{0,1\})$) 
such that 
\begin{itemize}
\item  $C\cap (X\times \{0\})=N_1\times \{0\}$, $C\cap (X\times \{1\})=N_2\times \{1\}$, and
\item The $N_1$ and $N_2$ framing coincides with the $C$ framing in $C\cap X\times \{0\}$ and $C\cap X\times \{1\}$, respectively. 
More precisely, if we identify $N_1\simeq N_1\times \{0\}$,
then the image of the $N_1$-framing vectors $T_i\in T(X\times \{0\})/TN_1$ in $T(X\times [0,1])/TC$ coincide
with the $C$-framing vectors, $i=1,\ldots, n$, and similarly for $N_2$.
\end{itemize}
Under a suitable choice of metric, in which $C$ intersects $X\times \{0,1\}$ orthogonally, the framing vectors of $C$, $N_1$, $N_2$ can be identified with normal vectors that coincide in
$(N_1\times \{0\}\cup N_2\times \{1\})\subseteq C$.
Being framed cobordant is an equivalence relation for neat framed $(m-n)$-submanifolds of $X$.

One simple version of the Pontryagin-Thom construction yields a $1-1$ correspondence between the cohomotopy set $[X,S^n]$ of a boundaryless
smooth compact manifold $X$, and the class of neat framed submanifolds of $X$ of codimension $n$, framed cobordant in $X$.
This correspondence assigns to $[\varphi]\in [X,S^n]$ the framed cobordism class of 
$(\varphi^{-1}(v), \varphi^*(e_i)_i)$ where $e_1,\ldots, e_n$ is a basis of $T_v S^n$ and $\varphi$ is 
a smooth representative of $[\varphi]$ transverse to $v\in S^n$. The cobordism class is independent of the choice of $\varphi$, $v$ and
the basis $(e_1,\ldots, e_n)$ of $T_v S^n$. 
A smooth homotopy $F: X\times[0,1]\to S^n$ transverse to $v$ between two representatives $\varphi$ and $\psi$ of $[\varphi]$ 
corresponds to the framed cobordism $F^{-1}(v)\subseteq X\times [0,1]$ between $\varphi^{-1}(v)$ and $\psi^{-1}(v)$ 
with all framings induced by the corresponding maps. Further, any neat framed submanifold of codimension $n$ is
the preimage $\varphi^{-1}(v)$ for some $\varphi: X\to S^n$ transverse to $v$ that induces the given framing,
see~\cite[Chapter 7]{Milnor:97} for details.

We will use a slight variation on this.
Let $X$ be a smooth compact $m$-manifold with (possibly non-empty) boundary $\manbound X$ and $A$ be closed. 
Let $v\neq *$ be two points in $S^n$. 
Our correspondence assigns to any smooth map $\varphi: (X,A)\to (S^n, *)$ such that $\varphi$ and $\varphi|_{\manbound X}$ 
are transverse to $v\in S^n$ the framed manifold $(\varphi^{-1}(v), \varphi^*(e_i)_i)$ where $e_i\in T_v S^n$ 
form a basis of the tangent space: note that this framed submanifold is disjoint from $A$.
A~homotopy $F: (X\times [0,1], A\times [0,1])\to (S^n, *)$ between $\varphi$ and $\psi$, such that $F$ and $F|_{\manbound X\times [0,1]}$
are both transverse to $v$ induces a framed cobordism $(F^{-1}(v), F^*(e_i)_i)$ between 
$(\varphi^{-1}(v), (\varphi^*(e_i))_i)$ and $(\psi^{-1}(v), (\psi^*(e_i))_i)$ that is disjoint from $A\times [0,1]$
and any such framed cobordism can be realized in this way.
Similarly, any framed submanifold of $X$ of codimension $n$ that is disjoint from $A$ can be realized as 
$(\varphi^{-1}(v), (\varphi^*(e_i))_i)$ for some $\varphi: (X,A)\to (S^n. *)$ such that $\varphi$ and $\varphi|_{\manbound X}$
are transverse to $v$. Similarly, any framed cobordism disjoint from $A\times [0,1]$ can be realized as 
the $v$-preimage of a smooth homotopy transverse to $v$.

Summarizing the above, there is a $1-1$ correspondence between $[(X,A), (S^n, *)]$ and framed cobordism classes of framed 
submanifolds of $X$ of codimension $n$, disjoint from $A$ via framed cobordisms in $X\times [0,1]$ disjoint from $A\times [0,1]$.

\heading{$\ZrFr(f)$ is a subset of a framed cobordism class.}
Assume that $f: X\to\R^n$ is such that $A=|f|^{-1}[r,\infty)$ and let $h$ be a~regular homotopy perturbation of $f$.
Then $h^{-1}(0)$ is disjoint from $A$ and $(h^{-1}(0), h^*(e_i)_i)$ is a framed manifold. 

Let $F: (X\times [0,1], A\times [0,1])\to (\R^n, \R^n\setminus\{0\})$ be the homotopy between $f$ and $h$.
After multiplying $F$ by a suitable positive scalar function that equals $1$ in a neighborhood of $F^{-1}(0)$, 
we get a homotopy 
$$
F': (X\times [0,1], A\times [0,1])\to (\R^n, \{x:\,|x|\geq r\})
$$
between $f$ and $h'$ so that $h'$ coincide with $h$ in a neighborhood of $h^{-1}(0)$.
The composition of $F'$ with the quotient map 
$$q: (\R^n, \{x:\,|x|\geq r\}) \to (\R^n/\{x:\,|x|\geq r\}, \{x:\,|x|\geq r\})\simeq (S^n, *)$$
gives a homotopy between $q f$ and $q h'$
as functions $(X,A) \to (S^n, *)$. The framed manifold $(h^{-1}(0), h^*(e_i))$ coincides with $((q h)^{-1}(v), (qh)^*(\tilde{e_i}))$ 
where $v=q(0)$ and $\tilde{e}_i=q_* e_i$ is the basis of the tangent space $T_v S^n$. 
It follows that $h^{-1}(0)$ is framed cobordant to $f^{-1}(0)$ by our relative version of the Pontryagin-Thom construction.

To summarize the above, any element of $\ZrFr(f)$ can be expressed as $h^{-1}(0)$ for some regular homotopy perturbation 
$h: X\to\R^n$ and the framed cobordism class induced by $h^{-1}(0)$ can be identified with $[qh']  = [h'_{/A}] = [f_{/A}]$ 
which is in the image of $\delta$. 

\heading{Surjectivity.}
Any homotopy class in the image of $\delta$ has a representative $f_{/A}$ that comes from a map $f: X\to\R^n$
that is nonzero on $A$ and after multiplying $f$ by a suitable scalar valued function, we may achieve that
$A=|f|^{-1}[r,\infty)$. Then $\ZrFr(f)$ is a framed cobordism class corresponding to $[f_{/A}]\in\mathrm{Im}(\delta)$.
This implies that the map $\ZrFr(f)\mapsto [f_{/A}]$ from Theorem C is surjective.

%
%
\heading{From cobordism to homotopy perturbations.}
To complete the proof, we will show that $\ZrFr(f)$ is the full cobordism class corresponding to $[f_{/A}]$. This also implies the
injectivity of the assignment $\ZrFr(f)\mapsto [f_{/A}]$ from Theorem C.

Thus we need to show that whenever $g$ is a regular homotopy perturbation of $f$ and $Z$ is a framed manifold
framed cobordant to $g^{-1}(0)$ via a framed cobordism disjoint from $A\times [0,1]$, then $Z$ can be realized as a zero set of some
regular homotopy perturbation $h$ of $f$. 

In what follows, we will exploit the constraint $m\leq 2n-3$ for the first time.
\begin{proof}[Rest of the proof of Theorem C]
Let $C$ be the framed cobordism between the neat $(m-n)$-submanifolds $g^{-1}(0)$ and $Z$, disjoint from $A\times [0,1]$,
such that $C\cap (X\times \{0\})=g^{-1}(0)\times \{0\}$. With a slight abuse of notation, we will use the 
identification $X\simeq X\times \{0\}$ and write $g: X\times \{0\}\to \R^n$.

Our goal is to construct a smooth map $F: X\times [0,1]\to\R^n$ such that $F$ and $F|_{\manbound X}$ are transverse to $0$ and 
the zero set $F^{-1}(0)$ is the framed cobordism $C$.
Using Lemma~\ref{l:product_n}, there exists a~product neighborhood $\mathcal{N}_1$ of $C$ diffeomorphic to $C\times B^{n}$ 
where $B^n$ is the closed unit $n$-ball, and a function $F: \mathcal{N}_1\to\R^n$ so that its framed zero set coincides with $C$.

We claim that there exists some $\epsilon>0$ such that the sub-neighborhood $\mathcal{N}\simeq C\times B^n(\epsilon)$ 
satisfies that $F|_{\partial{N}\cap (X\times \{0\})}$ and $g|_{\partial\mathcal{N}\cap (X\times \{0\})}$ 
are homotopic as maps to $\R^n\setminus \{0\}$. This is because $F$ and $g$ have equal differentials on $g^{-1}(0)$
and if $x$ is close enough to $g^{-1}(0)$, then the straight-line between $F(x)$ and $g(x)$ avoids zero.

On $\partial\mathcal{N}$, $F$ has values in $\R^n\setminus\{0\}$ and we want to extend it to a function $X\to\R^n$ that
is nonzero on $X\setminus\mathcal{N}$. 
Let $\mathcal{N}_0:=\mathcal{N}\cap (X\times \{0\})$. Exploiting that $F$ is homotopic to $g$ on $\partial\mathcal{N}_0$
and $g$ is nowhere zero on $\overline{(X\times \{0\}) \setminus \mathcal{N}_0}$, 
it follows that $F|_{\partial\mathcal{N}}$ can be extended to a nowhere zero function 
$$
\partial\mathcal{N} \cup ((\overline{X\times \{0\})\setminus \mathcal{N}_0})\to \R^n\setminus \{0\}.
$$
We want to show that it can be extended to a function $\overline{(X\times [0,1])\setminus \mathcal{N}}\to\R^n\setminus \{0\}$.
To simplify notation, let 
$U:=\overline{(X\times [0,1])\setminus \mathcal{N}}$ and 
$V:=\partial\mathcal{N} \cup ((\overline{X\times \{0\})\setminus \mathcal{N}_0})$ 
All spaces here are smooth compact manifolds (possibly with corners) and can be triangulated.
$\R^n\setminus \{0\}$ can equivalently be replaced by a~sphere $S^{n-1}$,
so we may apply obstruction theory to show that there are no obstructions to extendability. 
Assume any triangulation of the spaces and assume that $F$ has been extended to the $(k-1)$-skeleton $U^{(k-1)}\cup V$.
We want to show that there is no obstruction in extending it to the $k$-skeleton $U^{(k)}\cup V$. 
This obstruction is represented via an element of the cohomology group $H^k(U,V,\pi_{k-1}(S^{n-1}))$
and if this obstruction cohomology class vanishes, then there is an extension to the $k$-skeleton after possibly
changing the extension to the $(k-1)$-skeleton, see~\cite[Thm. 3.4]{prasolov} for details. 
We will show that the relevant cohomology groups are all trivial.

\begin{figure}
\begin{center}
\includegraphics{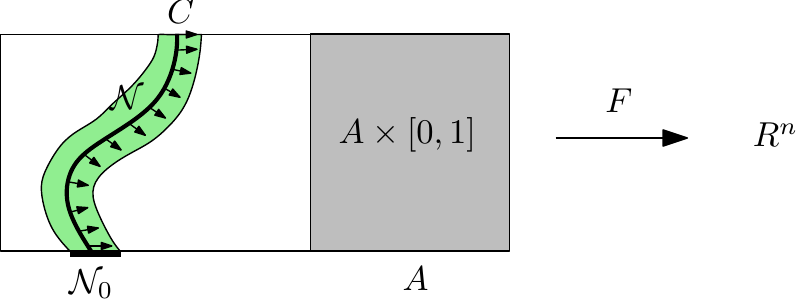}
\caption{Illustration of our construction. The function $F$ is defined so that $F^{-1}(0)$ is the framed cobordism $C$ and $\mathcal{N}$
is the product neighborhood of $C$.}
\label{fig:cob_to_hom}
\end{center}
\end{figure} 
The inclusion
$$
(U,V)=
(\overline{X\times [0,1]\setminus \mathcal{N}} , \partial\mathcal{N}\cup \overline{((X\times \{0\})\setminus\mathcal{N}_0))} 
\hookrightarrow (X\times [0,1], \mathcal{N}\cup (X\times \{0\}))
$$
induces cohomology isomorphism by excision.\footnote{We excise the interior of $\mathcal{N}$. The closure of this is $\mathcal{N}$ which is
not contained in the interior of $\mathcal{N}\cup (X\times \{0\})$, so we cannot use the common excision theorem in singular cohomology.
However, all the spaces here are smooth manifolds possibly with corners, and $X$ can be triangulated so that
$X_1:=\overline{X\times [0,1]\setminus\mathcal{N}}$ and $X_2:=(X\times \{0\})\cup\mathcal{N}$ are simplicial subcomplexes of $X$.
By \cite[Lemma 7, p. 190]{Spanier}, such pair $(X_1,X_2)$ of subcomplexes is an ``excisive couple'' and 
$(X_1,X_1\cap X_2)\hookrightarrow (X_1\cup X_2, X_2)$ induces (co)homology isomorphism.
} The long exact sequence of the triple
$(X\times [0,1], \mathcal{N}\cup (X\times \{0\}), X\times \{0\})$
implies that 
$$H^{k-1}(\mathcal{N}\cup (X\times \{0\}),\,X\times \{0\}, \pi) \simeq H^k(X\times [0,1], \mathcal{N}\cup (X\times \{0\}),\pi),$$
because $H^*(X\times [0,1], X\times \{0\}, \pi)$ is trivial with any coefficient group $\pi$.

Using another excision, this time excising $(X\times \{0\})\setminus \mathcal{N}_0$, we may replace 
$H^{k-1}(\mathcal{N}\cup (X\times \{0\}),\,X\times \{0\}, \pi)$ by $H^{k-1}(\mathcal{N}, \mathcal{N}_0, \pi)$.
This last pair can further be replaced by $H^{k-1}(C,C_0,\pi)$ where $C$ is the cobordism and $C_0:=C\cap (X\times \{0\})$.

We want to show that $H^{k-1}(C,C_0, \pi_k(S^{n-1}))$ is trivial for all $k$. For $k<n-1$, it follows from the triviality of 
$\pi_k(S^{n-1})$. For $k\geq n-1$, we note that $\dim C=m-n+1$ and the relative cohomology of $(C,C_0)$ is trivial in dimension
exceeding $m-n+1$. The constraint $m\leq 2n-3$ can be rewritten to $m-n+1< n-1$, so any $k\geq n-1$ is larger than $\dim C$.

Trivality of the obstruction groups implies that $F|_{\partial\mathcal{N}}$ can be extended to a map 
$\overline{(X\times [0,1])\setminus \mathcal{N}}\to\R^n\setminus\{0\}$.
Extending it further via the already defined map on $\mathcal{N}$, we obtain $F: X\times [0,1]\to \R^n$ that is smooth in a neighborhood of $C$,
its zero set is $C$ and it induces the framing of $C$. The map $h:=F(\cdot, 1)$ is smooth in a neighborhood of $Z$, induces the framing of
$Z$, and $(h/|h|)|_A$ is homotopic to $(g/|g|)_A$ via the homotopy $F|_{A\times [0,1]}$.
Possibly replacing $h$ with a perturbation that is smooth everywhere and unchanged in a neighborhood of $Z$
yields the desired homotopy perturbation of $g$.
\end{proof}

\heading{Acknowledgments.} The research leading to these results has received funding from Austrian Science Fund (FWF): M 1980,
the People Programme (Marie Curie Actions) of the European Unionâs Seventh Framework Programme
(FP7/2007-2013) under REA grant agreement number [291734] and from the Czech Science Foundation
(GACR) grant number 15-14484S with institutional support RVO:67985807.
The research by Marek Kr\v cá\'al was also supported by the project 17-09142S of GA \v{C}R.
We are grateful to Sergey Avvakumov, Ulrich Bauer, Marek Filakovsk\'y, Amit Patel,  
Luk\'a\v s Vok\v r\'inek and Ryan Budnay for useful discussions and hints.
\bibliography{../sratscha,../Postnikov}

\appendix \newpage
\section{Rest of the proof of Theorem A.}
\label{s:rest_thm_A}

\heading{Case $Z_{\leq r}(f)$, part (1).}
Here we give the proof of part (1) of Theorem~\ref{t:hom-class} for the more technical case of zero sets of non-strict perturbations $Z_{\leq r}(f)$.
We will call functions $g$ with $\|g-f\|\leq r$ \emph{non-strict $r$-perturbations} of $f$ and functions $g$ for which $\|g-f\|<r$ \emph{strict} $r$-perturbations.

Assume that $B=|f|^{-1}(r)$ is given in addition to $A=|f|^{-1}[r,\infty)$. These two subspaces of $X$ determine $W:=|f|^{-1}[0,r]$
which can be expressed as $(X\setminus A)\cup B$. The zero set of each non-strict $r$-perturbation is contained in $W$.
Let us denote by $M$ the mapping cylinder of the inclusion $A \hookrightarrow X$, that is, $M:=X\times\{0\}\cup A\times [0,1]$
with the identifications $A\simeq A\times \{1\}\subseteq M$ and $X\simeq X\times \{0\}\subseteq M$.
Statement $(1)$ of Theorem~\ref{t:hom-class} is an immediate consequence of the following lemma.
\begin{lemma}
\label{l:nonstrict}
$Z_{\leq r}(f)$ equals to the family
$$
\{h^{-1}(0)\mid h:M\to\R^n\text{ s.t. } h|_{A\times\{1\}}=f|_A
\text{ and } h^{-1}(0)\subseteq W\times \{0\}\}.$$
\end{lemma}
The maps $h\:M\to\R^n$ as above will be called \emph{homotopy perturbations of $f$.} 
If $g\sim f$ are homotopic as maps $A\to \R^n\setminus \{0\}$, then the family of zero sets 
of homotopy perturbations of $g$ is equal to the family of zero sets of homotopy perturbations of $f$.  
\begin{proof}[Proof of Lemma~\ref{l:nonstrict}]
First assume that a map $g: X\to\R^n$ satisfies $\|g-f\|\leq r$. 
Then the map $h\:M\to\R^n$ that is equal to $g$ on $X\simeq X\times \{0\}$ and to the straight line homotopy 
between $g$ and $f$ on $A\times[0,1]$ is a homotopy perturbation of $f$ with $h^{-1}(0)=g^{-1}(0)$: this homotopy
is clearly nonzero on $A\times (0,1]$ and $h^{-1}(0)\subseteq W\simeq W\times \{0\}$.

Conversely, assume that a homotopy perturbation $h\:M\to\R^n$ of $f$ is given.
We will denote by $h'$ the restriction $h|_X$. Let us define $O_j:=|h'|^{-1}[0,1/j)$.  These sets are open neighborhoods of $h^{-1}(0)$ in $X$, the intersection of all $O_j$ is the zero set $h^{-1}(0)$ and $\bar{O}_{j+1}\subseteq O_j$ (consequently $\bar{O}_{j+1}$ is disjoint from $X\setminus O_j$).
Let as define a partial homotopy $G_1'$ of $h'|_{X\setminus O_2}$ on $(A\setminus O_1)\cup \partial O_2$ as follows.
We define $G_1'$ to be equal to $h$ on $(A\setminus O_1)\times [0,1]$ and to be the stationary homotopy equal to $h'$ on $\partial O_2$
(that is, $G_1'(o,t)=h'(o)$ for $o\in \partial O_2$ and $t\in [0,1]$).

The partial homotopy $|G_1'|$ is bounded from below and above by positive constants
$m$ and $M$, so we can define the target space of $G_1'$ to be the triangulated space $T=\{x\in\R^n:\,m\leq |x|\leq M\}$. 
The homotopy extension property of the pair 
$(X\setminus O_2, (A\setminus O_1)\cup \partial O_2)$ with respect to $T$ implies that $G_1'$ can be extended to a nowhere zero map
$G_1: (X\setminus O_2)\times[0,1]\to T$ such that $G_1(\cdot,0)=h'|_{X\setminus O_2}$.

Inductively, we define homotopies $G_j: (X\setminus O_{j+1})\times[0,1]\to\R^n\setminus \{0\}$ such that
\begin{itemize}
\item $G_j$ equals $G_{j-1}$ on $(X\setminus O_{j-1})\times [0,1]$,
\item $G_j=h$ on $(A\setminus O_j)\times [0,1]$, 
\item $G_j$ is the stationary homotopy equal to $h'|_{\partial O_{j+1}}$ on $\partial O_{j+1}$, and
\item $G_j(0)=h'|_{X\setminus O_{j+1}}$.
\end{itemize}
Let $G_j'$ be a partial homotopy of $h'|_{X\setminus O_{j+1}}$ on $(X\setminus O_{j-1})\cup (A\setminus O_{j})\cup (\partial O_{j+1})$
defined by the first three properties of $G_j$ above. 
This is well defined and continuous, because $G_{j-1}$ equals $h$ on $(A\setminus O_{j-1})\times [0,1]$, and $\partial O_{j+1}$ is disjoint from
the other two parts. 
By the homotopy extension property, there exists
a homotopy $G_j: (X\setminus O_{j+1})\times[0,1]\to\R^n\setminus\{0\}$ satisfying all four properties above.

Let as define maps $g_j: X\to\R^n$ by 
$$
g_j(x)=\begin{cases}
G_j(x,1)\quad\text{ for $x\in X\setminus O_{j+1}$ },\\
h'(x)\quad\text{ for $x\in \bar{O}_{j+1}$.}
\end{cases}
$$
These maps are continuous and satisfy
\begin{itemize}
\item $g_j=g_{j-1}$ outside $O_{j-1}$,
\item $g_j^{-1}(0)=h^{-1}(0)$,
\item $g_j$ is an extension of $f|_{A\setminus O_j}$.
\end{itemize}

For each $j>1$ we define the constant 
$$c_j:=\frac{j-1}{j} \,\min\{1, \frac{r}{\max_{\overline{O}_j} |f|}\}.$$
The zero set $h^{-1}(0)$ is contained in $W$ where $|f|\leq r$, so the maximum of $|f|$ on the shrinking neighborhoods
$\overline{O}_j$ of $h^{-1}(0)$ converge to $\leq r$ and it follows that $c_j<1$ and $c_j\to 1$ as $j\to\infty$.

Let $\alpha_j: X\to [c_j,1]$ be so that $\alpha_j=1$ outside $O_{j-1}$ and $\alpha_j=c_j$ inside $O_j$.
Define $f_j\:X\to\R^n$ by $f_j:=\alpha_j\,f$. We have that $|f_j|<r$ on $O_j\cup (X\setminus A)$ and hence 
$|f_j|^{-1}[r, \infty)\subseteq A\setminus O_j$. Further, $\|f_j-f\|\to 0$ follows from $\alpha_j\in [c_j, 1]$.

The map $\alpha_j g_j$ is an extension of $f_j|_{A\setminus O_j}$ and hence an extension of $f_j|_{|f_j|^{-1}[r, \infty)}$, 
so by Lemma~\ref{l:perturb-ext}, some positive scalar multiple
$\beta_j g_j$ of $g_j$ is a strict $r$-perturbation of $f_j$. We will show that $\beta_j\:X\to(0,1]$ may be chosen so that they additionally satisfy
\begin{itemize}
\item $\beta_j=\beta_{j-1}$ outside $O_{j-1}$ (and hence $\beta_j g_j=\beta_{j-1} g_{j-1}$ outside $O_{j-1}$), 
\item $|\beta_j g_j|\leq \frac{1}{j}$ in $\bar{O}_{j}$, and
\item $|\beta_j g_j|\leq |\beta_{j-1}\,g_{j-1}|$ on $X\setminus O_j$.
\end{itemize}

Assume that such $\beta_1,\ldots, \beta_{j-1}$ have been chosen. 
Because $g_{j}=g_{j-1}$ and $f_{j}=f_{j-1}$ outside $O_{j-1}$, we have
$\beta_{j-1} g_{j}=\beta_{j-1} g_{j-1}$ and thus $\beta_{j-1}\,g_j$ is a strict $r$-perturbation of $f_{j}$ in $X\setminus O_{j-1}$. 
If $\beta_{j}^{'}$ is so that $\beta_{j}'\,g_{j}$ is a global strict $r$-perturbation of $f_{j}$, 
we may define $\beta_{j}''$ to be a positive scalar extension of $\beta_{j}'$ in $\bar{O}_{j}$ and 
of $\beta_{j-1}$ on $X\setminus O_{j-1}$. Then $\beta_j'' g_j$
is a strict $r$-perturbation of $f_{j}$ on $\bar{O}_{j} \cup X\setminus O_{j-1}$.
Furthermore, $\beta_{j}''\,g_{j}$ is a strict $r$-perturbation of $f_{j}$ on some open neighborhood $U$ of $X\setminus O_{j-1}$.
By multiplying $\beta_{j}''$ with a $(0,1]$-valued function that equals $1$ on $X\setminus O_{j-1}$, and is small enough in
${O}_{j-1}\setminus U$, we get a positive function $\beta_j'''$ such that $\beta_{j}'''\leq \beta_{j}'$ in ${O}_{j-1}\setminus U$ 
and that $|\beta_{j}'''\,g_{j}|\leq \frac{1}{j}$ in $\bar{O}_{j}$.  
The resulting $\beta_{j}'''\,g_{j}$ is still a strict $r$-perturbation of $f_{j}$ on $X$ since each $\beta'''_j(x)g_j(x)$ is a strict convex combination of $\beta'_j(x)g_j(x)$ (less than $r$-far from $f_j(x)$) and $0$ (at most $r$-far from $f_j(x)$).
Finally, we multiply $\beta_j'''$ 
by some positive extension $X\to (0,1]$ of the function $\min \{1, |\beta_{j-1} g_{j-1}|/|\beta'''_{j} g_{j}|\}$ defined on $X\setminus O_j$
to get the desired function $\beta_j$ and then $\beta_j\,g_j$ is a strict $r$-perturbation of $f_j$ 
satisfying all the three properties above.

Let $g(x):=\lim_j \beta_j(x)\,g_j(x)$ for all $x\in X$. We will show that this is well defined and continuous.
If $h(x)\neq 0$, then some neighborhood $U(x)$ of $x$ is contained in $X\setminus O_j$ for $j$ large enough
and for any $y\in U(x)$, $\beta_j(y) g_j(y)=\beta_{j+1}(y)\,g_{j+1}(y)=\ldots$ is stabilized. 
Further, if $h(x)=0$ then for each $j$, some neighborhood of $x$ is contained in $O_j$ and $|\beta_i\,g_i|\leq 1/j$ for each $i>j$ on this neighborhood.
This shows that $g(x)=0$ and $g$ is continuous in $x$. 

By construction, $g^{-1}(0)=h^{-1}(0)$ and the inequality $|\beta_j(x)\,g_j(x) - f_j(x)|<r$ implies that $|g(x)-f(x)|\leq r$ holds for each $x\in X.$
 \end{proof}

\heading{Case $Z_{\leq r}(f)$, part (2).}
The proof of the second part of Theorem~\ref{t:hom-class} is similar to the $<r$ part and the $m\leq 2n-4$ case 
immediately follows from this analog of Lemma~\ref{l:stronger}:
\begin{lemma}
\label{l:stronger<=}
Let $B\subseteq A\subseteq X$ be cell complexes, $m:=\dim X\leq 2n-3$, $r>0$ and let $f_1, f_2: X\to\R^n$ be such that 
$A=|f_1|^{-1}[r,\infty)=|f_2|^{-1}[r, \infty)$, $B=|f_1|^{-1}(r)=|f_2|^{-1}(r)$. Assume further that $A$
has dimension at most $2n-4$ and
let $f_1\boxplus f_2$ be a~representant of the sum of $[f_1]$ and $[f_2]$ in the group $[A,S^{n-1}]$. Then 
$$
Z_{\leq r}(f_1\boxplus f_2)\supseteq \{Z_1\cup Z_2:\,\,Z_1\in Z_{\leq r}(f_1)\,\, and\,\, Z_2\in Z_{\leq r}(f_2) \}
$$
\end{lemma}
We already know that if $A,B$ are given, then $Z_{\leq r}(f_1\boxplus f_2)$ depends only on the homotopy class of $[f_1\boxplus f_2]$ 
so the left hand side is well defined.
\begin{proof}
Let $Z_i\in Z_{\leq r}(f_i)$ and $M=A\times [0,1]\cup X\times \{0\}$. By Lemma~\ref{l:nonstrict} there exist $h_1, h_2: M\to\R^n$ 
such that $(h_i)|_{A\times \{1\}}=(f_i)|_A$ and $h_i^{-1}(0)=Z_i$ for $i=1,2$. Let $Z:=Z_1\cup Z_2$ 
and $\bar h_i: M\setminus Z\to S^{n-1}$ be defined by $h_i/|h_i|$. Similarly as in the proof of Lemma~\ref{l:stronger},
we choose a sequence of cell complexes $Y_0\subseteq Y_1\subseteq \ldots \subseteq M$ such that $Y:=\cup Y_i=M\setminus Z$
and $Y_{i-1}$ is a subcomplex of $Y_{i}$, and cellular approximations $a_i: Y_i\to S^{n-1}\times S^{n-1}$ of $(\bar h_1, \bar h_2)$.
Such cellular approximation is possible, because $\dim M\leq 2n-3<2n-2$, although its homotopy class may depend on the choice of
the approximation. However, the homotopy class of the restriction $a_i|_{A}$ is already well defined due to the constraint $\dim A\leq 2n-4$. 
The composition of $a_i$ and the folding map $\nabla: S^{n-1}\vee S^{n-1}\to S^{n-1}$ define a chain of maps
$\nabla a_i: Y_i\to S^{n-1}$ and $\bar h:=\cup \nabla a_i$ is map $M\setminus Z\to S^{n-1}$ such that
$\bar h|_{A\times \{1\}}\sim (\bar f_1)|_A\boxplus (\bar f_2)|_A$. 
Multiplying $\bar h$ by the distance function $\mathrm{dist}(Z,\cdot)$
we obtain a suitable $h: M\to\R^n$ which is a~homotopy perturbation of $f_1\boxplus f_2$ and its zero set is contained in $Z_{\leq f}(f_1\boxplus f_2)$
by Lemma~\ref{l:nonstrict}.
 \end{proof}

If $m\leq 2n-4$ then it follows from the last Lemma that for any $f_2\in i^*[X,S^{n-1}]$---which implies 
$\emptyset\in Z_{\leq r}(f_2)$---we have that $Z_{\leq r}(f_1\boxplus f_2)\supseteq Z_{\leq r}(f_1)$ and the extendability of
$\boxminus f_2$ then implies 
$$
Z_{\leq r}(f_1)=Z_{\leq r}(f_1\boxplus f_2\boxminus f_2)\supseteq Z_{\leq r}(f_1\boxplus f_2)
$$
which proves part (2) of Theorem~\ref{t:hom-class}.

In the case $m=2n-3$ we again take $A'=\partial A$ which has dimension at most $2n-4$, $X'=\overline{X\setminus A}$ and
$f': X'\to\R^n$ to be the restriction $f|_{X'}$. The homotopy class $[f_{/A}]$ determines the homotopy class of $[f'_{/A'}]$
and the following lemma implies that $Z_{\leq r}(f)$ is uniquely determined by $Z_{\leq r}(f')$, $A$ and $B$. This will complete the proof
of Theorem~\ref{t:hom-class}.
\begin{lemma}
\label{l:f'/A'->f/A}
Let $X$ be a cell complex, $f: X\to\R^n$ and $X',A',A$ and $B$ be defined as above. Then
$$
Z_{\leq r}(f)=\{U\cup V:\,U\in Z_{\leq r}(f')\,\,\text{and}\,\,V\subseteq B \text{ is closed }\}.
$$
\end{lemma}
\begin{proof}
If $Z=g^{-1}(0)$ for some non-strict $r$-perturbation $g$ of $f$ then $Z=U\cup V$ where $U$ is the zero set of the restriction $g|_{X'}$
and $V:=B\cap g^{-1}(0)$ is closed. 

Conversely, let $U=g'^{-1}(0)$ for some non-strict $r$-perturbation $g'$ of $f'$ and let $V\subseteq B$ be closed.
Let $M=A\times [0,1]\cup X\times \{0\}$ be the mapping cylinder of $A\subseteq X$ and $M'=A'\times [0,1]\cup X'\times \{0\}$ its subset.
By Lemma~\ref{l:nonstrict} there exists $h': M'\to\R^n$ such that $h'|_{A'}=f'|_{A'}$ and $U=h'^{-1}(0)$.
Let $\bar h':=h'/|h'|$ be the sphere valued map $M'\setminus U\to S^{n-1}$. 

We will show that it can be extended to a map $\bar h: M\setminus (U\cup V)\to S^{n-1}$ such that 
$\bar h|_{A\times \{1\}}=\bar f|_A$. 
To define $\bar h$ on $((A\setminus A')\times [0,1])\setminus V$, will use a zig-zag sequence of partial extensions. 
Let $\{O_j\}_j$ be a collection of open neighborhoods of $U\cup V$ in $X$ such that 
$\overline{O}_{j+1}\subseteq O_j$ and $\cap_j O_j=U\cup V$. 
\begin{figure}
\begin{center}
\includegraphics{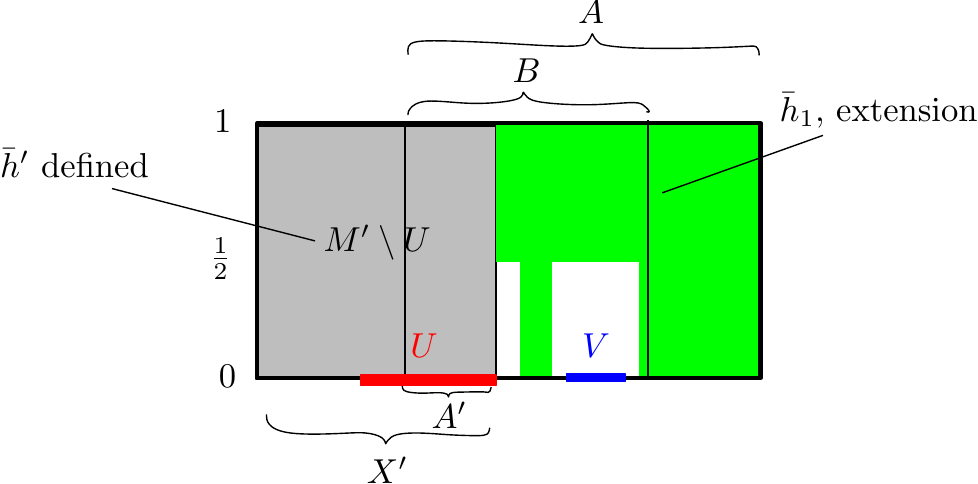}
\caption{Illustration of the sequence of partial extensions in Lemma~\ref{l:f'/A'->f/A}.}
\label{fig:aux_lem}
\end{center}
\end{figure}

Let 
$$\bar h_1': (M'\setminus U)\cup (A\times [\frac{1}{2}, 1])\to S^{n-1}$$ be an extension of 
$\bar h'$ such that 
$\bar h_1'|_{A\times \{1\}}=\bar f|_{A\times \{1\}}$ and let 
$$\bar h_1: (M'\setminus U)\cup (A\times [\frac{1}{2}, 1])\cup ((X\setminus O_1)\times [0,1])\to S^{n-1}$$
be a further extension of $\bar h_1'$: these extensions exist by the homotopy extension property.
Inductively, define $\bar h'_{j}$ to be a map 
$(M'\setminus U)\cup (A\times [\frac{1}{j+1}, 1])\cup ((X\setminus O_{j-1})\times [0,1])\to S^{n-1}$
that extends $\bar h_{j-1}$ and $\bar h_j$ be an extension of this map that is also defined on $(X\setminus O_j)\times [0,1]$
where it extends $\bar h'|_{(A'\setminus O_j)\times [0,1]}$ and $\bar h_{j}'$.
The union $\cup \bar h_j$ is the desired map 
$\bar h\: M\setminus (U\cup V)\to S^{n-1}$. Multiplying $\bar h$ by the scalar function $\mathrm{dist}(U\cup V, \cdot): M\to\R_0^+$
we obtain a homotopy perturbation of $f$ with the desired zero set and it follows from Lemma~\ref{l:nonstrict} that
$U\cup V\in Z_{\leq r}(f)$.
 
\end{proof}

\end{document}